\pdfoutput=1

\documentclass[11pt, a4paper]{article}

\usepackage{hyperref}
\usepackage[left=3cm, right=3cm, bottom=2.5cm]{geometry}

\usepackage[utf8]{inputenc}
\usepackage[T1]{fontenc}
\usepackage{lmodern}
\usepackage{microtype}
\usepackage{float}
\usepackage{graphicx}

\usepackage{amsmath,amssymb,amsthm}
\usepackage{cleveref}
\usepackage[dvipsnames]{xcolor}
\usepackage{paralist}
\usepackage{bm}
\usepackage{dutchcal}

\usepackage{titling}
\newif\ifkeepslowthings
\keepslowthingstrue
\definecolor{cb-yellow}{RGB}{221,170,51}
\definecolor{cb-red} {RGB}{187,85,102}
\definecolor{cb-teal}{RGB}{0,153,136}
\definecolor{cb-blue} {RGB}{0,68,136}
\definecolor{cb-green}{RGB}{17,119,51}
\definecolor{cb-purple} {RGB}{170,68,153}
\definecolor{cb-palegrey} {RGB}{221,221,221}

\usepackage{tikz,pgfplots}
\usetikzlibrary{positioning}
\pgfplotsset{compat=1.18}
\hypersetup{
  pdfauthor={Maximilian Wiesmann},
  pdftitle={Lee--Yang phenomena in edge-coloured graph counting},
  pdfsubject={},
  pdfkeywords={},
    colorlinks=true,
    linkcolor=cb-red,
    filecolor=magenta,      
    urlcolor=cb-purple,
    citecolor=cb-yellow
}

\newtheorem{theorem}{Theorem}
\numberwithin{theorem}{section}
\newtheorem{corollary}[theorem]{Corollary}
\newtheorem{proposition}[theorem]{Proposition}

\newtheorem{definition}[theorem]{Definition}
\newtheorem*{maintheorem}{Main Theorem}

\theoremstyle{remark}
\newtheorem{remark}[theorem]{Remark}
\newtheorem{examplex}[theorem]{Example}
\newenvironment{example}
{\pushQED{\qed}\examplex}
{\popQED\endexamplex}

\newcommand{\GG}{\mathcal{G}}

\newcommand{\RR}{\mathbb{R}}
\newcommand{\QQ}{\mathbb{Q}}

\newcommand{\CC}{\mathbb{C}}
\newcommand{\ZZ}{\mathbb{Z}}

\newcommand{\per}{\mathrm{per}}
\DeclareMathOperator{\dlog}{dlog}

\DeclareMathOperator{\crit}{crit}
\DeclareMathOperator{\Aut}{Aut}
\DeclareMathOperator{\Hess}{Hess}
\DeclareMathOperator{\IS}{InSec}

\renewcommand{\d}{\mathrm{d}}

\renewcommand{\L}{\mathcal{L}}

\newcommand{\J}{\mathcal{J}}
\newcommand{\h}{\mathcal{h}}
\newcommand{\V}{\mathcal{V}}
\newcommand{\cS}{\mathcal{S}}
\newcommand{\cP}{\mathcal{P}}

\renewcommand{\Re}{\mathfrak{R}}
\renewcommand{\Im}{\mathfrak{I}}
\newcommand{\g}{\mathfrak{g}}

\newcommand{\bb}[1]{\boldsymbol{#1}}
\newcommand{\Sym}{\mathfrak{S}}
\newcommand{\lf}{\mathrm{lf}}
\newcommand{\Sc}{\mathbb{S}}
\renewcommand{\i}{\mathrm{i}}

\newcommand\restr[2]{{% we make the whole thing an ordinary symbol
  \left.\kern-\nulldelimiterspace % automatically resize the bar with \right
  #1 % the function
  \littletaller % pretend it's a little taller at normal size
  \right|_{#2} % this is the delimiter
  }}
\newcommand{\littletaller}{\mathchoice{\vphantom{\big|}}{}{}{}}

\theoremstyle{theorem}
\newtheorem*{theorem*}{Theorem}

\usepackage[backend=biber,style=alphabetic,maxnames=99]{biblatex}
\DefineBibliographyStrings{english}{
  in = {}
}
\AtBeginBibliography{\small}

\title{\bf Lee--Yang phenomena in \\edge-coloured graph counting}
\date{}
\author{Maximilian Wiesmann}

\begin{document}

\maketitle

\vspace{-3em}
\begin{abstract}
  We study the accumulation of zeros of a polynomial arising from the enumeration of edge-coloured graphs along certain limit curves. The polynomial is a variant of an edge-chromatic polynomial, which specialises to the partition function of the ferromagnetic Ising model on a random regular graph. We call this accumulation behaviour a Lee--Yang phenomenon in analogy with the Lee--Yang theorem. The limiting loci are semialgebraic and arise from anti-Stokes curves of an exponential integral. 
  % In the Ising model case, our approach provides a unified explanation of Lee--Yang and Fisher zeros.
\end{abstract}

% \begin{abstract}
%   We study the asymptotic distribution of zeros of a univariate polynomial arising from the enumeration of edge-coloured regular graphs. The polynomial is a variant of an edge-chromatic polynomial and, for a specific choice of weights, specialises to the averaged ferromagnetic Ising partition function on random regular graphs. We call the resulting accumulation behaviour a Lee–Yang phenomenon, in analogy with the Lee–Yang theorem. We show that the limiting loci are semialgebraic curves arising as anti-Stokes curves of an associated exponential integral.
% \end{abstract}

\section{Introduction}

In this paper we study the univariate polynomial
\begin{equation}
  \label{eq:A^V_n_intro}
  A^V_n(\lambda) := \sum_{G \in \GG_{-n}} \frac{1}{|\Aut(G)|}\prod_{v\in V_G} \Lambda_{\deg(v)}(\lambda) \in \CC[\lambda].
\end{equation}
Here, $\GG_{-n}$ denotes the set of isomorphism classes of edge-coloured graphs $G$ with Euler characteristic $\chi(G) = |V_G| - |E_G| = -n$, and $\Lambda_{(w_1,\dots,w_d)}(\lambda)$ are vertex markings which associate to each vertex that is incident to $w_i$ half-edges of the $i^{\text{th}}$ colour a univariate polynomial in $\lambda$ (where $d$ is the number of edge colours). We assume that for $\bb{w}\in\ZZ^d_{\geq 0}$, only finitely many $\Lambda_{\bb{w}}$ are nonzero, and we conveniently store them in the polynomial
\[
  V(\bb{x}, \lambda) = \sum_{\bb{w}\in \ZZ^d_{\geq 0},\, |\bb{w}| \geq 1} \Lambda_{\bb{w}}(\lambda) \frac{\bb{x}^{\bb{w}}}{\bb{w}!} \in \CC[\lambda][x_1,\dots,x_d].
\]
Throughout the paper, we consider the case where $V$ is homogeneous of degree $k$ in $\bb{x}$. This restricts the sum in \eqref{eq:A^V_n_intro} to a sum over certain $k$-regular graphs.\par

The polynomial $A^V_n$ can be regarded as a variant of an edge-chromatic polynomial. Moreover, for a certain choice of $V$, $A^V_n$ is an average partition function of the ferromagnetic Ising model on a regular graph. This statement is made precise in Section \ref{sec:Ising}.\par

Our main focus is on the zeros of $A^V_n(\lambda)$ for complex $\lambda$. As $n$ increases, one observes a rather flummoxing accumulation of the zeros along certain limit curves, as in Figure \ref{fig:roots}. 
\begin{figure}[H]
  \centering
  \includegraphics[width=0.6\textwidth]{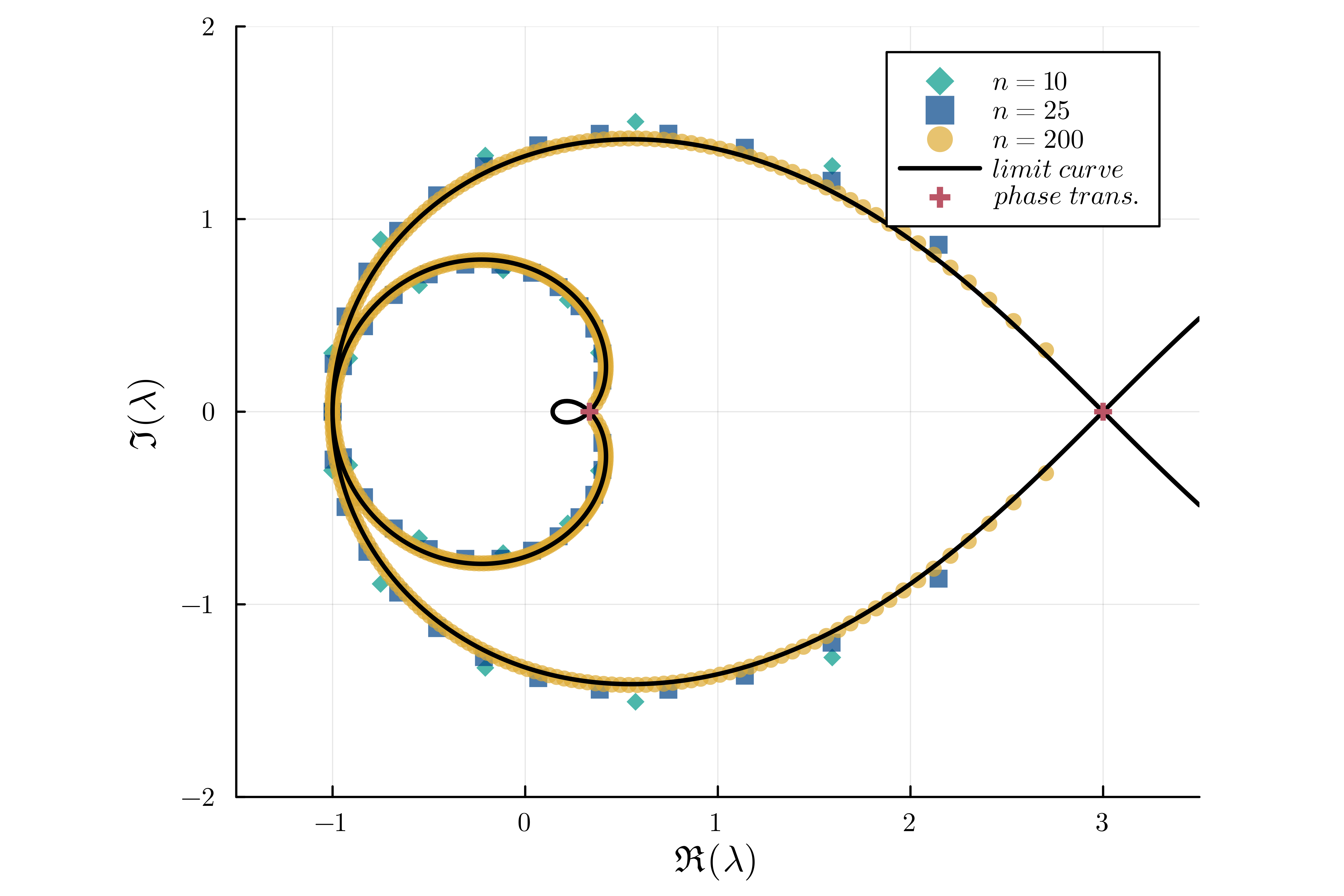}
  \caption{The roots of $A^V_n(\lambda)$ in the complex $\lambda$-plane, for $V(x_1,x_2,\lambda) = \frac{x_1^4}{4!} + \lambda \frac{x_1^2x_2^2}{2!\cdot 2!} + \lambda^2 \frac{x_2^4}{4!}$.}
  \label{fig:roots}
\end{figure}
The red crosses denote phase transitions (points of non-analyticity) of the asymptotics of $A^V_n(\lambda)$ for $n\rightarrow \infty.$ The goal of this paper is the description of the limit curves. Below we state an informal version of our main result, Theorem \ref{thm:main}.

\begin{maintheorem}[informal]
  Under mild non-degeneracy assumptions on $V$, the zeros of $A^V_n$ accumulate along parts of anti-Stokes curves as $n\rightarrow\infty$, except for possibly finitely many isolated points. 
  Those limit curves are given as follows. Let $\Sc = \{\bb{x}\in \CC^d \,:\, x_1^2+\dots + x_d^2 = 1\}$ and let $\crit_{\Sc}(V)$ be the set of critical points $\sigma$ of $V$ restricted to $\Sc$ such that $V(\sigma)\neq 0$. 
  Then, for $\sigma,\sigma'\in\crit_{\Sc}(V)$, an anti-Stokes curve is given by the condition
  \[
    \left\{ \lambda\in\CC \,:\, \Re \log \restr{V(\bb{x},\lambda)}{\bb{x}=\sigma} = \Re \log \restr{V(\bb{x},\lambda)}{\bb{x}=\sigma'} \right\}.
  \]
\end{maintheorem}

The accumulation of zeros is reminiscent of the celebrated Lee--Yang theorem \cite{leeYang} from statistical physics. It states that the zeros of the ferromagnetic Ising model partition function on any finite graph all lie on the imaginary axis, when the partition function is viewed as a function in a magnetic field parameter. 
As the system size increases, the set of accumulation points is exactly the imaginary axis. The importance of this theorem stems from the philosophy that partition function zeros reveal the location of phase transitions \cite{biskup2004partition}. 
Since the 50s, the theory has been generalised in many directions, for example by considering zeros in different physical parameters, e.g.\ in the complex temperature leading to Fisher zeros \cite{fisher1965lectures}, or by extending the class of lattice models \cite{lieb1981general}. A very general and rigorous treatment of partition function zeros has been given in \cite{biskup2004partition}. \par   

Lee--Yang theory has also received a lot of interest from combinatorialists. In \cite{borcea2009lee} the authors use stable polynomials to prove the Lee--Yang theorem. 
% A characterisation of Lee--Yang polynomials (whose roots lie on the unit circle) was given in \cite{ruelle2010characterization}. 
The zeros of chromatic polynomials (a.k.a.\ Potts model partition functions) and their accumulation have been studied in \cite{sokal2004chromatic}: Sokal constructs a countable family of graphs such that the roots of their chromatic polynomials are dense in the complex plane, except for a small disc. \par 

The key step in proving the main theorem is the exponential integral representation
\begin{equation}
    \label{eq:intro_An_exp_integral}
    A^V_n (\lambda) = \frac{2^{nM + (d-2)/2} \left(nM +\frac{d-2}{2}\right)!}{(2\pi)^{d/2} (nK)!} \int_{S^{d-1}} V(\bb{x},\lambda)^{nK} \,\omega_{S^{d-1}},
\end{equation}
where $nK$ and $nM$ are integers depending on the regularity $k$ of the graphs (Proposition \ref{prop:An_exp_integral}). For real parameters $\lambda$, the limit $\lim_{n\rightarrow\infty} A^V_n(\lambda)$ can simply be taken by performing a Laplace expansion around the maxima of $|V(\bb{x},\lambda)|$ on $S^{d-1}$. However, for complex $\lambda$, an analytic continuation of \eqref{eq:intro_An_exp_integral} by means of Picard--Lefschetz theory becomes necessary in the large-$n$ limit. Such an analytic continuation does not exist globally, but on certain domains $D \subset \CC$. On each such domain, the integral from \eqref{eq:intro_An_exp_integral} can be written as
\begin{equation}
  \label{eq:intro_Lefschetz_decomp}
  \sum_{\sigma \in \mathrm{crit}_{\Sc}(V)} c_\sigma(\lambda) \int_{\Gamma_\sigma} V(\bb{x},\lambda)^{nK} \,\omega_{S^{d-1}},
\end{equation}
where $\Gamma_\sigma$ are steepest descent contours associated to each critical point, also known as Lefschetz thimbles, and $c_\sigma(\lambda)$ are complex coefficients. Here, one passes to the complex compactification $\Sc$ of the real sphere $S^{d-1}$. Each integral in \eqref{eq:intro_Lefschetz_decomp} admits a stationary phase approximation that asymptotically reduces to an evaluation at the critical point in the $n\rightarrow\infty$ limit. 
The coefficients $c_\sigma(\lambda)$ encode topological information and are only analytic within each domain $D$ but can jump if one leaves $D$. This is referred to as the Stokes' phenomenon, see e.g.\ \cite{pham1985descente}. Each domain $D$ is a connected component of the complement of the arrangement of Stokes curves of the integral in \eqref{eq:intro_An_exp_integral}. \par

To derive the decomposition \eqref{eq:intro_Lefschetz_decomp}, one needs to show that the Lefschetz thimbles form a basis of the appropriate homology group. Here, a technical difficulty arises because for many $V$ of interest (in particular, the ones related to the Ising model considered in Section \ref{sec:Ising}), several critical points yield the same evaluation of $V$, thus violating a commonly found assumption in Picard--Lefschetz theory, see e.g.\ \cite[Hypothesis H4]{delabaere2002global}. 
To remedy this problem we use a construction by Matsubara \cite{matsubara2023localization} to find a basis of Lefschetz thimbles and derive the expression \eqref{eq:intro_Lefschetz_decomp}. 
This is done in Section \ref{sec:Lefschetz}, culminating in Theorem \ref{thm:lefschetz_decomp}. The stationary phase formula then gives the asymptotic expression (Corollary \ref{cor:stationary_phase})
\begin{equation}
  \label{eq:intro_stationary_phase}
  \int_{S^{d-1}} V(\bb{x},\lambda)^{n} \,\omega_{S^{d-1}} \sim  \sum_{\sigma\in\crit_{\Sc}(V)}\tilde{c}_{\sigma}(\lambda) e^{n \log \restr{V(\bb{x},\lambda)}{\bb{x}=\sigma}} (1+ o(n^{-1})) \quad (n\to\infty).
\end{equation} 
A sum decomposition such as \eqref{eq:intro_stationary_phase} is the starting point for the works \cite{sokal2004chromatic,biskup2004partition}. In the physics context, the quantities $\log \restr{V(\bb{x},\lambda)}{\bb{x}=\sigma}$ can be thought of as metastable free energies. Zeros then accumulate along regions of multiphase coexistence, meaning that several free energies have the same magnitude.
Indeed, we can apply a generalisation of the Beraha--Kahane--Weiss theorem due to Sokal \cite[Theorem 1.5]{sokal2004chromatic} to conclude that the zeros of $A^V_n(\lambda)$ accumulate along (parts of) the curves defined by
\begin{equation}
  \label{eq:intro_limit_curve}
  \Re \log \restr{V(\bb{x},\lambda)}{\bb{x}=\sigma} = \Re \log \restr{V(\bb{x},\lambda)}{\bb{x}=\sigma'} \quad \left( \Leftrightarrow |V(\sigma,\lambda)| = |V(\sigma^\prime,\lambda)| \right),
\end{equation}
where $\sigma$ and $\sigma^\prime$ are critical points of $V$ on $\Sc$. Here, we need to slightly modify Sokal's result to accommodate for the fact that analyticity of the coefficients $\tilde{c}_\sigma(\lambda)$ in \eqref{eq:intro_stationary_phase} can only be guaranteed in a one-sided neighbourhood of the accumulation points (Section \ref{sec:accumulation}). \par

\begin{figure}
  \centering
  \includegraphics[width=0.7\textwidth]{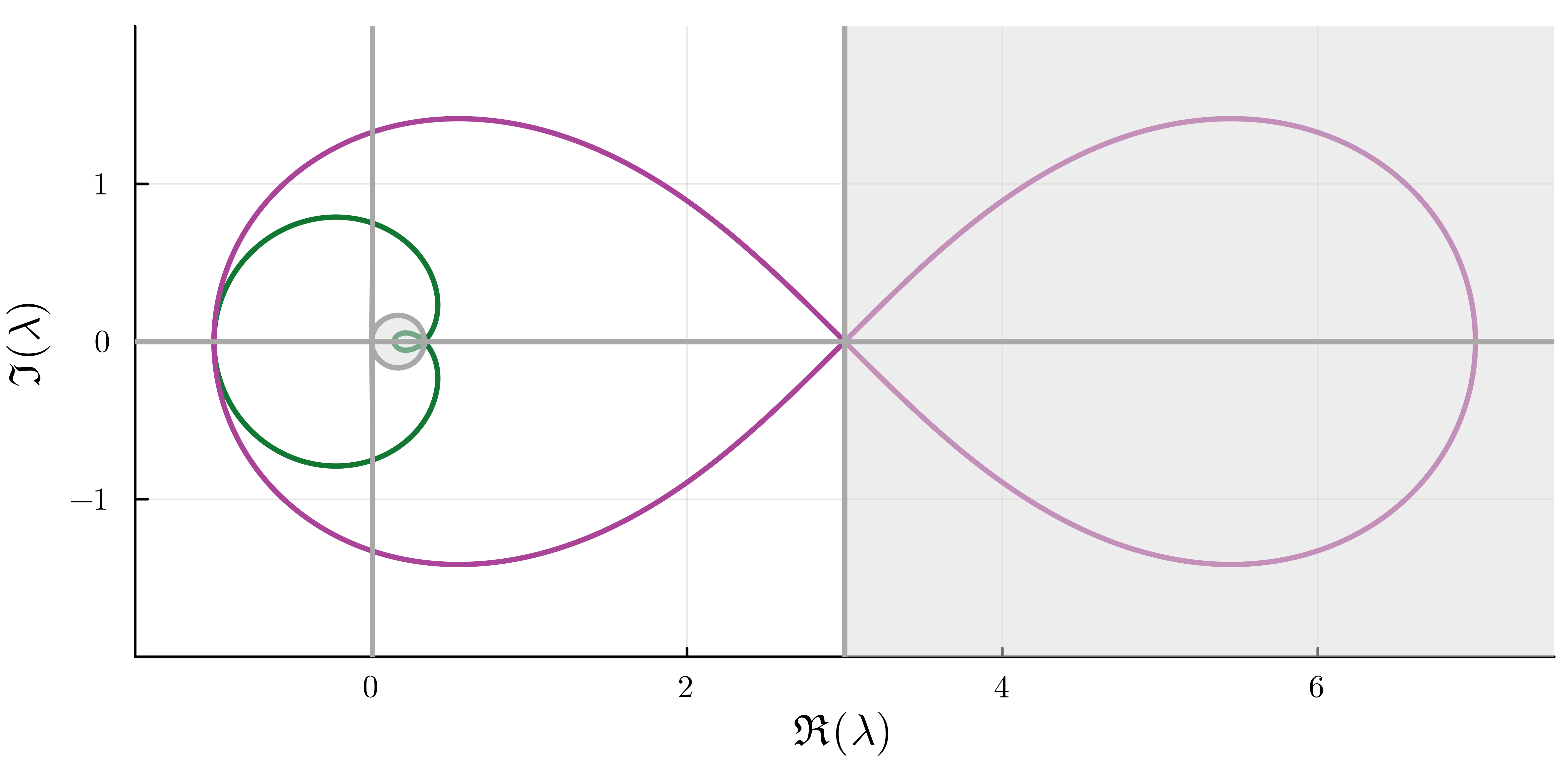}
  \caption{Anti-Stokes curves (green and purple) and Stokes curves (grey) for the example depicted in Figure \ref{fig:roots}. The roots accumulate along the intersection of the anti-Stokes curves with certain regions in the complement of the Stokes curves (non-shaded regions).}
  \label{fig:Stokes_and_anti_Stokes}
\end{figure}

The curve defined in \eqref{eq:intro_limit_curve} is known as an anti-Stokes curves of the integral \eqref{eq:intro_An_exp_integral}, i.e.\ a curve where the contribution from different thimble integrals from \eqref{eq:intro_Lefschetz_decomp} exchange dominance.
As mentioned before, the zeros only accumulate along parts of the curves defined by \eqref{eq:intro_limit_curve}. This comes from the fact that \eqref{eq:intro_Lefschetz_decomp} is not globally valid. 
In fact, not necessarily all critical points $\sigma\in \crit_{\Sc}(V)$ contribute to \eqref{eq:intro_Lefschetz_decomp} for a given $\lambda$. The loci where the $c_\sigma$ (and $\tilde{c}_\sigma$) can potentially become zero are given by the Stokes curves defined by 
\begin{equation}
  \label{eq:intro_Stokes_curves}
  \Im \restr{V(\bb{x},\lambda)}{\bb{x}=\sigma}  = \arg V(\sigma,\lambda) \overset{!}{=} \arg V(\sigma^\prime,\lambda) = \Im \restr{V(\bb{x},\lambda)}{\bb{x}=\sigma} .
\end{equation} 
The accumulation of zeros is only along the intersection of the curves defined by \eqref{eq:intro_limit_curve} with those regions in the complement of the curves defined by \eqref{eq:intro_Stokes_curves} where both thimbles that exchange dominance on the anti-Stokes curve also contribute to the integral \eqref{eq:intro_Lefschetz_decomp}. 
This is illustrated in Figure \ref{fig:Stokes_and_anti_Stokes} for the example from Figure \ref{fig:roots}. The green and purple curves are the anti-Stokes curves, the grey curves are the Stokes curves. The accumulation as in Figure \ref{fig:roots} only happens along the intersection of the anti-Stokes curves with the non-shaded regions.\par

The idea of a relation between Lee--Yang (or Fisher) zeros and anti-Stokes curves has appeared before in different physics contexts: \cite{kanazawa2015structure} observes such a relation numerically for a partition function in quantum field theory, and \cite{itzykson1983distribution,dolan2001thin,janke2001fat} use the idea to study zeros of Ising model partition functions.
The paper \cite{dolan2001thin} should be highlighted here as it studies partition function zeros for the Ising model on random 3- and 4-regular graphs, thus being closely related to our discussion in Section \ref{sec:Ising}. \par

\textbf{Outline of the article.} In Section \ref{sec:graphs} we explain the framework of half-edge labelled coloured graphs, how the polynomial $A_n^V$ naturally arises in the enumeration of such graphs (Corollary 2.4), and we prove the representation of $A_n^V$ as an exponential integral (Proposition \ref{prop:An_exp_integral}). Section \ref{sec:Lefschetz} is devoted to the Lefschetz thimble construction. We recall the construction from \cite{matsubara2023localization} and adapt it to our context. The main results are the existence of a basis of Lefschetz thimbles (Theorem \ref{thm:lefschetz_decomp}) and the resulting asymptotic expression (Corollary \ref{cor:stationary_phase}). The proof of the main theorem (Theorem \ref{thm:main}) is completed in Section \ref{sec:accumulation}. In the final Section \ref{sec:Ising} we explain how our framework can be seen as a unified approach to Lee--Yang and Fisher zeros of the Ising model on a random regular graph.\par

\textbf{Acknowledgments.} I am very grateful to Saiei Matsubara for pointing me to \cite{matsubara2023localization} and his patient and helpful explanations on Lefschetz thimbles. I also thank Michael Borinsky, Mario Kummer, Christian Sevenheck and Bernd Sturmfels for valuable discussions.

\section{Edge-coloured graphs}
\label{sec:graphs}

In this section we explain the construction of edge-coloured graphs via half-edge labels. This also appears in \cite{multicolor}, where we use it to obtain asymptotic numbers of proper edge-colourings. In the case of bicoloured graphs, the construction has already been described in \cite{BMW}.
For us, a graph $G$ is a finite, one-dimensional CW complex, i.e.\ we allow for self-loops and multiple edges. Although in combinatorics it might be more common to define graphs as simplicial complexes, the use of CW complexes is very natural, both in mathematics or physics. 
In particular, in the context of quantum field theory, Feynman graphs are naturally identified as CW complexes \cite{bessis1980quantum}. 
Additionally, to each edge of the graph we assign one of $d$ colours.
To encode $G$ using only discrete data we use half-edge labelled graphs.
This leads to the following definition.

\begin{definition}
  \label{def:multicolored_half_edge_label}
  Let $H_1,\dots,H_d$ be disjoint sets of half-edge labels, one for each colour. An $[H_1,\dots,H_d]$-\emph{half-edge labelled coloured graph} $\Gamma$ is a tuple $\Gamma = (V, E_{H_1},\dots, E_{H_d})$ where 
  \begin{enumerate}
    \item $V$, the set of vertices, is a set partition of $H_1 \sqcup \dots \sqcup H_d$;
    \item for all $i=1,\dots,d$, $E_{H_i}$ is a set partition of $H_i$ into blocks of size two.
  \end{enumerate}
\end{definition}

An edge-coloured graph $G$ is identified with an isomorphism class $[\Gamma]$ of a half-edge labelled coloured representative $\Gamma$. Here, an \emph{isomorphism} $j$ between $\Gamma = (V, E_{H_1},\dots, E_{H_d})$ and $\Gamma^\prime = (V^\prime, E^\prime_{H^\prime_1},\dots, E^\prime_{H^\prime_d})$ is a tuple $j = (j_1,\dots,j_d)$ of bijections $j_i\colon H_i \rightarrow H^\prime_i$ such that the canonically induced maps $(j_i)_\ast$ on the partitions satisfy $(j_i)_\ast (E_{H_i}) = E^\prime_{H^\prime_i}$ for all $i=1,\dots,d$, and $(j_1\sqcup\dots\sqcup j_d)_\ast (V) = V^\prime$. An \emph{automorphism} of $\Gamma$ is an isomorphism from $\Gamma$ to itself. By $|\Aut(G)|$ we denote the cardinality of the automorphism group of $\Gamma$, for some representative $G = [\Gamma]$. The set of all isomorphism classes of half-edge labelled coloured graphs is denoted by $\GG$. From now on, we might drop the adjective ``edge-coloured''.\par 
The set of vertices of $G$ is denoted by $V_G$, while the set of edges is $E_G$. Below there is an example of how to encode a graph via a half-edge labelled graph.

\begin{example}
    Let $H_1 = \{1,2,\dots,6\}$ and $H_2 = \{a, b\}$. The partitions 
    \begin{align*}
        V & = \left\{ \{1,2,3,4\}, \{5,6,a,b\} \right\}, \\
        E_{H_1} & = \left\{ \{1,2\}, \{3,4\}, \{5,6\} \right\}, \\
        E_{H_2} & = \left\{ \{a,b\} \right\},
    \end{align*}
    form an $[H_1,H_2]$-labelled graph representing the graph $G$ depicted in Figure \ref{fig:extwotadpole2s_rrry}. Its automorphism group is isomorphic to $(\Sym_2\times \Sym_2 \rtimes \Sym_2) \times \Sym_2 \times \Sym_2$.
\end{example}

\begin{figure}
    \centering
    \scalebox{0.5}{\begin{tikzpicture}[x=1ex,y=1ex,baseline={([yshift=-1.5ex]current bounding box.center)}]
    \coordinate (v);

    \coordinate [left=8 of v] (vm1);
    \coordinate [left=3 of vm1] (v01);
    \coordinate [right=3 of vm1] (v11);
    \draw[cb-red, line width = 3.5] (v01) circle(3);
    \draw[cb-red, line width = 3.5] (v11) circle(3);
    \filldraw (vm1) circle (3pt);

    \coordinate [right=8 of v] (vm2);
    \coordinate [left=3 of vm2] (v02);
    \coordinate [right=3 of vm2] (v12);
    \draw[cb-red, line width = 3.5] (v02) circle(3);
    \draw[cb-yellow, line width = 3.5] (v12) circle(3);
    \filldraw (vm2) circle (3pt);
\end{tikzpicture}%
}
    \caption{An edge-bicoloured graph $G$ with two connected components.}
\label{fig:extwotadpole2s_rrry}
\end{figure}
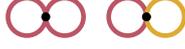

A graph $G$ is equipped with a \emph{degree function} $\deg\colon V_G \rightarrow \ZZ^d_{\geq 0}$, where $\deg(v) = (w_1,\dots,w_d)$ if the vertex $v$ is incident to $w_i = |v\cap H_i|$ half-edges of colour $i$, for $i=1,\dots,d$. \par   
In the following we make use of some multi-index notation. Bold symbols denote $d$-tuples, as in $\bb{w} = (w_1,\dots,w_d)$. The multi-index factorial is $\bb{w}! = w_1!\dots w_d!$, while $\bb{x}^{\bb{w}}$ is shorthand notation for $x_1^{w_1} \dots x_d^{w_d}$. Moreover, we use $|\bb{w}| = w_1+\dots +w_d$.

\begin{proposition}
  \label{prop:gf_graphs}
  The generating function for graphs with marked vertex degrees is 
  \begin{equation}
    \label{eq:generating_function}
    \sum_{G \in \GG} \frac{\eta^{|E_G|}}{|\Aut(G)|}\prod_{v\in V_G} \Lambda_{\deg(v)} =  \sum_{\bb{s} \geq 0} \eta^{|\bb{s}|} (2s_1-1)!!\cdots (2s_d-1)!!\cdot [\bb{x}^{2\bb{s}}]\exp\left( \sum_{\substack{\bb{w}\in \ZZ^d_{\geq 0}\\ |\bb{w}| \geq 1}} \Lambda_{\bb{w}} \frac{\bb{x}^{\bb{w}}}{\bb{w}!} \right).
  \end{equation}
  Here, $[\bb{x}^{2\bb{s}}]$ denotes the coefficient extraction operator, and \eqref{eq:generating_function} is an expression in the formal power series ring $\QQ[\Lambda_{\bb{w}} \,:\, \bb{w}\in\ZZ^d_{\geq 0}][[\eta]]$ with coefficients in the formal variables $\Lambda_{\bb{w}}$.
\end{proposition}

\begin{proof}
  By expanding the exponential on the right-hand side, we can write
  \begin{equation}
    \label{eq:exponential_and_block_partitions}
    (2 \bb{s})! \, [x_1^{2s_1}\dots x_d^{2s_d}]\exp\left( \sum_{\substack{\bb{w}\in \ZZ^d_{\geq 0}\\ |\bb{w}| \geq 1}} \Lambda_{\bb{w}} \frac{\bb{x}^{\bb{w}}}{\bb{w}!} \right)
    = \sum_{\{n_{\bb{w}}\}} \frac{(2 \bb{s})!}{\prod_{\bb{w}} n_{\bb{w}}!\, \bb{w}!^{n_{\bb{w}}}} \prod_{\substack{\bb{w}\in \ZZ^d_{\geq 0}\\ |\bb{w}| \geq 1}} \Lambda_{\bb{w}}^{n_{\bb{w}}},
  \end{equation}
  where the sum on the right-hand side ranges over all assignments $\bb{w} \mapsto n_{\bb{w}} \in \ZZ_{\geq 0}$ with the property $\sum_{\bb{w}} n_{\bb{w}} w_i = 2s_i$ for all $i=1,\dots,d$. The term 
  \begin{equation}
    \label{eq:block_partitions}
    \frac{(2 \bb{s})!}{\prod_{\bb{w}} n_{\bb{w}}!\, \bb{w}!^{n_{\bb{w}}}}
  \end{equation}
  counts the number of partitions of the disjoint union $H_1 \sqcup \dots \sqcup H_d$ into $n_{\bb{w}}$ many blocks with $w_i$ elements from $H_i$, for $i=1,\dots,d$. Here, $H_i$ is assumed to have $2s_i$ many elements. This can be seen as follows: the $d$-fold product of symmetric groups $\Sym_{2s_1} \times \dots \times \Sym_{2s_d}$ acts transitively on the set of partitions with specified block structure. A partition with $n_{\bb{w}}$ many blocks with $w_i$ elements from $H_i$ is stabilised by a subgroup isomorphic to $(\Sym_{w_1} \times \dots \times \Sym_{w_d})^{n_{\bb{w}}} \rtimes \Sym_{n_{\bb{w}}}$, permuting the elements inside each block and the blocks themselves. By the orbit-stabiliser theorem, \eqref{eq:block_partitions} is the number claimed above. \par   
  The number of partitions of $H_i$ into blocks of size two is equal to the number of matchings in the complete graph on $H_i$, which is $(2s_i - 1)!!$. Choosing a partition of $H_1\sqcup\dots\sqcup H_d$ and matchings of each $H_i$ specifies a half-edge labelled graph representing a graph $G$ with $|\bb{s}|$ many edges. 
  The number of such half-edge labelled graphs representing $G$ equals $(2\bb{s})!/|\Aut(G)|$,
  again by an application of the orbit-stabiliser theorem. Multiplying \eqref{eq:exponential_and_block_partitions} by $\eta^{|\bb{s}|} (2s_1-1)!!\dots (2s_d-1)!!$ and summing over all $\bb{s}$ proves the proposition.
\end{proof}

It is convenient to shift the generating function such that the index of summation is the \emph{Euler characteristic} $\chi(G)$ of $G$. Recall that $\chi(G) = |V_G| - |E_G|$.

\begin{corollary}
    \label{cor:gen_fun}
    Let $\GG_{-n}$ be the set of isomorphism classes of half-edge labelled graphs with Euler characteristic $-n$. Then the following identity holds true.
    \begin{equation}
    \label{eq:gen_fun_euler}
    \sum_{\substack{n\geq 0\\ G \in \GG_{-n}}} \frac{z^{\chi(G)}}{|\Aut(G)|}\prod_{v\in V_G} \Lambda_{\deg(v)} =  \sum_{\bb{s} \geq 0} z^{|\bb{s}|} (2s_1-1)!!\cdots (2s_d-1)!! [\bb{x}^{2\bb{s}}]\exp\left( \sum_{\substack{\bb{w}\in \ZZ^d_{\geq 0}\\ |\bb{w}| \geq 1}} z\Lambda_{\bb{w}} \frac{\bb{x}^{\bb{w}}}{\bb{w}!} \right)
  \end{equation}  
\end{corollary}

The expression \eqref{eq:gen_fun_euler} is an equality in the formal power series ring $\mathcal{R}[[z]]$ over the polynomial ring $\mathcal{R} = \QQ[\Lambda_{\bb{w}}]$ in infinitely many variables. We will be interested in the case where all the variables $\Lambda_{\bb{w}}$ are specified to depend on only a \emph{single} parameter $\lambda$, and only finitely many are nonzero. In other words, we specialise all $\Lambda_{\bb{w}}$ in such a way so that
\[
    V(\bb{x},\lambda) = \sum_{\bb{w}\in \ZZ^d_{\geq 0},\, |\bb{w}| \geq 1} \Lambda_{\bb{w}}(\lambda) \frac{\bb{x}^{\bb{w}}}{\bb{w}!}
\]
is a polynomial in $\QQ[\lambda][\bb{x}]$. Then the left-hand side of \eqref{eq:gen_fun_euler} becomes a generating function for a specific set of graphs with vertex degree markings and allowed vertex-incidences determined by the coefficients and nonzero terms of $V(\bb{x},\lambda)$. Each coefficient in \eqref{eq:gen_fun_euler},
\[
  A^V_n(\lambda) := \sum_{G \in \GG_{-n}} \frac{1}{|\Aut(G)|}\prod_{v\in V_G} \Lambda_{\deg(v)}(\lambda),
\]
is now a univariate polynomial in $\lambda$.

\begin{example}
  \label{ex:ising_expansion}
    Consider the case of two colours, $d=2$, and fix the polynomial 
    \begin{equation}
        \label{eq:V_Ising}
        V(x_1,x_2,\lambda) = \frac{x_1^4}{4!} + \lambda \frac{x_1^2 x_2^2}{2!\cdot 2!} + \lambda^2 \frac{x_2^4}{4!}.
    \end{equation}
    This restricts the generating function \eqref{eq:gen_fun_euler} to 4-regular graphs where each vertex is incident to four half-edges which are all the same colour, or to two half-edges of the first and two half-edges of the second colour. The parameter $\lambda$ counts the number of edges of the second colour.
    For example, the Euler characteristic $-n=-2$ contribution to \eqref{eq:gen_fun_euler} is given by 
    \begin{align*}
        & \scalebox{.3}{\begin{tikzpicture}[x=1ex,y=1ex,baseline={([yshift=-1.5ex]current bounding box.center)}]
    \coordinate (v);

    \coordinate [above=4 of v] (vm1);
    \coordinate [left=3 of vm1] (v01);
    \coordinate [right=3 of vm1] (v11);
    \draw[cb-red, line width = 3.5] (v01) circle(3);
    \draw[cb-red, line width = 3.5] (v11) circle(3);
    \filldraw (vm1) circle (3pt);

    \coordinate [below=4 of v] (vm2);
    \coordinate [left=3 of vm2] (v02);
    \coordinate [right=3 of vm2] (v12);
    \draw[cb-red, line width = 3.5] (v02) circle(3);
    \draw[cb-red, line width = 3.5] (v12) circle(3);
    \filldraw (vm2) circle (3pt);
\end{tikzpicture}%
} \; +\;  \scalebox{.3}{\begin{tikzpicture}[x=1ex,y=1ex,baseline={([yshift=-1.5ex]current bounding box.center)}]
    \coordinate (vm);
    \coordinate [left=4 of vm] (v0);
    \coordinate [right=4 of vm] (v1);
    \draw[cb-red, line width = 3.5] (v0) to[bend left=45] (v1);
    \draw[cb-red, line width = 3.5] (v0) to[bend right=45] (v1);
    \draw[cb-red, line width = 3.5] (vm) circle(4);
    \filldraw (v0) circle (3pt);
    \filldraw (v1) circle (3pt);
\end{tikzpicture}%
} \;+\; \scalebox{.3}{\begin{tikzpicture}[x=1ex,y=1ex,baseline={([yshift=-1.5ex]current bounding box.center)}]
    \coordinate (vm);
    \coordinate [left=4 of vm] (v0);
    \coordinate [right=4 of vm] (v1);
    \coordinate [left=4 of v0] (vc1);
    \coordinate [right=4 of v1] (vc2);
    \draw[cb-red, line width = 3.5] (vc1) circle(4);
    \draw[cb-red, line width = 3.5] (vm) circle(4);
    \draw[cb-red, line width = 3.5] (vc2) circle(4);
    \filldraw (v0) circle (3pt);
    \filldraw (v1) circle (3pt);
\end{tikzpicture}%
} = \tfrac{1}{128} + \tfrac{1}{48} + \tfrac{1}{16} = \tfrac{35}{384}, \\
        & \scalebox{.3}{\begin{tikzpicture}[x=1ex,y=1ex,baseline={([yshift=-1.5ex]current bounding box.center)}]
    \coordinate (v);

    \coordinate [above=4 of v] (vm1);
    \coordinate [left=3 of vm1] (v01);
    \coordinate [right=3 of vm1] (v11);
    \draw[cb-red, line width = 3.5] (v01) circle(3);
    \draw[cb-red, line width = 3.5] (v11) circle(3);
    \filldraw (vm1) circle (3pt);

    \coordinate [below=4 of v] (vm2);
    \coordinate [left=3 of vm2] (v02);
    \coordinate [right=3 of vm2] (v12);
    \draw[cb-red, line width = 3.5] (v02) circle(3);
    \draw[cb-yellow, line width = 3.5] (v12) circle(3);
    \filldraw (vm2) circle (3pt);
\end{tikzpicture}%
} \;+\; \scalebox{.3}{\begin{tikzpicture}[x=1ex,y=1ex,baseline={([yshift=-1.5ex]current bounding box.center)}]
    \coordinate (vm);
    \coordinate [left=4 of vm] (v0);
    \coordinate [right=4 of vm] (v1);
    \coordinate [left=4 of v0] (vc1);
    \coordinate [right=4 of v1] (vc2);
    \draw[cb-red, line width = 3.5] (vc1) circle(4);
    \draw[cb-red, line width = 3.5] (vm) circle(4);
    \draw[cb-yellow, line width = 3.5] (vc2) circle(4);
    \filldraw (v0) circle (3pt);
    \filldraw (v1) circle (3pt);
\end{tikzpicture}%
} = \lambda \left( \tfrac{1}{32} + \tfrac{1}{8}\right) = \tfrac{5}{32} \lambda, \\
        & \scalebox{.3}{\begin{tikzpicture}[x=1ex,y=1ex,baseline={([yshift=-1.5ex]current bounding box.center)}]
    \coordinate (v);

    \coordinate [above=4 of v] (vm1);
    \coordinate [left=3 of vm1] (v01);
    \coordinate [right=3 of vm1] (v11);
    \draw[cb-red, line width = 3.5] (v01) circle(3);
    \draw[cb-red, line width = 3.5] (v11) circle(3);
    \filldraw (vm1) circle (3pt);

    \coordinate [below=4 of v] (vm2);
    \coordinate [left=3 of vm2] (v02);
    \coordinate [right=3 of vm2] (v12);
    \draw[cb-yellow, line width = 3.5] (v02) circle(3);
    \draw[cb-yellow, line width = 3.5] (v12) circle(3);
    \filldraw (vm2) circle (3pt);
\end{tikzpicture}%
} \; +\; \scalebox{.3}{\begin{tikzpicture}[x=1ex,y=1ex,baseline={([yshift=-1.5ex]current bounding box.center)}]
    \coordinate (v);

    \coordinate [above=4 of v] (vm1);
    \coordinate [left=3 of vm1] (v01);
    \coordinate [right=3 of vm1] (v11);
    \draw[cb-red, line width = 3.5] (v01) circle(3);
    \draw[cb-yellow, line width = 3.5] (v11) circle(3);
    \filldraw (vm1) circle (3pt);

    \coordinate [below=4 of v] (vm2);
    \coordinate [left=3 of vm2] (v02);
    \coordinate [right=3 of vm2] (v12);
    \draw[cb-red, line width = 3.5] (v02) circle(3);
    \draw[cb-yellow, line width = 3.5] (v12) circle(3);
    \filldraw (vm2) circle (3pt);
\end{tikzpicture}%
} \; +\;  \scalebox{.3}{\begin{tikzpicture}[x=1ex,y=1ex,baseline={([yshift=-1.5ex]current bounding box.center)}]
    \coordinate (vm);
    \coordinate [left=4 of vm] (v0);
    \coordinate [right=4 of vm] (v1);
    \draw[cb-yellow, line width = 3.5] (v0) to[bend left=45] (v1);
    \draw[cb-yellow, line width = 3.5] (v0) to[bend right=45] (v1);
    \draw[cb-red, line width = 3.5] (vm) circle(4);
    \filldraw (v0) circle (3pt);
    \filldraw (v1) circle (3pt);
\end{tikzpicture}%
} \;+\; \scalebox{.3}{\begin{tikzpicture}[x=1ex,y=1ex,baseline={([yshift=-1.5ex]current bounding box.center)}]
    \coordinate (vm);
    \coordinate [left=4 of vm] (v0);
    \coordinate [right=4 of vm] (v1);
    \coordinate [left=4 of v0] (vc1);
    \coordinate [right=4 of v1] (vc2);
    \draw[cb-red, line width = 3.5] (vc1) circle(4);
    \draw[cb-yellow, line width = 3.5] (vm) circle(4);
    \draw[cb-red, line width = 3.5] (vc2) circle(4);
    \filldraw (v0) circle (3pt);
    \filldraw (v1) circle (3pt);
\end{tikzpicture}%
} \;+\; \scalebox{.3}{\begin{tikzpicture}[x=1ex,y=1ex,baseline={([yshift=-1.5ex]current bounding box.center)}]
    \coordinate (vm);
    \coordinate [left=4 of vm] (v0);
    \coordinate [right=4 of vm] (v1);
    \coordinate [left=4 of v0] (vc1);
    \coordinate [right=4 of v1] (vc2);
    \draw[cb-yellow, line width = 3.5] (vc1) circle(4);
    \draw[cb-red, line width = 3.5] (vm) circle(4);
    \draw[cb-yellow, line width = 3.5] (vc2) circle(4);
    \filldraw (v0) circle (3pt);
    \filldraw (v1) circle (3pt);
\end{tikzpicture}%
} = \lambda^2 \left( \tfrac{1}{64} + \tfrac{1}{32} + \tfrac{1}{8} + \tfrac{1}{16} + \tfrac{1}{16}\right) = \tfrac{19}{64}\lambda^2, \\   
        & \scalebox{.3}{\begin{tikzpicture}[x=1ex,y=1ex,baseline={([yshift=-1.5ex]current bounding box.center)}]
    \coordinate (v);

    \coordinate [above=4 of v] (vm1);
    \coordinate [left=3 of vm1] (v01);
    \coordinate [right=3 of vm1] (v11);
    \draw[cb-red, line width = 3.5] (v01) circle(3);
    \draw[cb-yellow, line width = 3.5] (v11) circle(3);
    \filldraw (vm1) circle (3pt);

    \coordinate [below=4 of v] (vm2);
    \coordinate [left=3 of vm2] (v02);
    \coordinate [right=3 of vm2] (v12);
    \draw[cb-yellow, line width = 3.5] (v02) circle(3);
    \draw[cb-yellow, line width = 3.5] (v12) circle(3);
    \filldraw (vm2) circle (3pt);
\end{tikzpicture}%
} \;+\; \scalebox{.3}{\begin{tikzpicture}[x=1ex,y=1ex,baseline={([yshift=-1.5ex]current bounding box.center)}]
    \coordinate (vm);
    \coordinate [left=4 of vm] (v0);
    \coordinate [right=4 of vm] (v1);
    \coordinate [left=4 of v0] (vc1);
    \coordinate [right=4 of v1] (vc2);
    \draw[cb-red, line width = 3.5] (vc1) circle(4);
    \draw[cb-yellow, line width = 3.5] (vm) circle(4);
    \draw[cb-yellow, line width = 3.5] (vc2) circle(4);
    \filldraw (v0) circle (3pt);
    \filldraw (v1) circle (3pt);
\end{tikzpicture}%
} = \lambda^3\left( \tfrac{1}{32} + \tfrac{1}{8}\right) = \tfrac{5}{32} \lambda^3, \\
        & \scalebox{.3}{\begin{tikzpicture}[x=1ex,y=1ex,baseline={([yshift=-1.5ex]current bounding box.center)}]
    \coordinate (v);

    \coordinate [above=4 of v] (vm1);
    \coordinate [left=3 of vm1] (v01);
    \coordinate [right=3 of vm1] (v11);
    \draw[cb-yellow, line width = 3.5] (v01) circle(3);
    \draw[cb-yellow, line width = 3.5] (v11) circle(3);
    \filldraw (vm1) circle (3pt);

    \coordinate [below=4 of v] (vm2);
    \coordinate [left=3 of vm2] (v02);
    \coordinate [right=3 of vm2] (v12);
    \draw[cb-yellow, line width = 3.5] (v02) circle(3);
    \draw[cb-yellow, line width = 3.5] (v12) circle(3);
    \filldraw (vm2) circle (3pt);
\end{tikzpicture}%
} \; +\;  \scalebox{.3}{\begin{tikzpicture}[x=1ex,y=1ex,baseline={([yshift=-1.5ex]current bounding box.center)}]
    \coordinate (vm);
    \coordinate [left=4 of vm] (v0);
    \coordinate [right=4 of vm] (v1);
    \draw[cb-yellow, line width = 3.5] (v0) to[bend left=45] (v1);
    \draw[cb-yellow, line width = 3.5] (v0) to[bend right=45] (v1);
    \draw[cb-yellow, line width = 3.5] (vm) circle(4);
    \filldraw (v0) circle (3pt);
    \filldraw (v1) circle (3pt);
\end{tikzpicture}%
} \;+\; \scalebox{.3}{\begin{tikzpicture}[x=1ex,y=1ex,baseline={([yshift=-1.5ex]current bounding box.center)}]
    \coordinate (vm);
    \coordinate [left=4 of vm] (v0);
    \coordinate [right=4 of vm] (v1);
    \coordinate [left=4 of v0] (vc1);
    \coordinate [right=4 of v1] (vc2);
    \draw[cb-yellow, line width = 3.5] (vc1) circle(4);
    \draw[cb-yellow, line width = 3.5] (vm) circle(4);
    \draw[cb-yellow, line width = 3.5] (vc2) circle(4);
    \filldraw (v0) circle (3pt);
    \filldraw (v1) circle (3pt);
\end{tikzpicture}%
} = \lambda^4\left( \tfrac{1}{128} + \tfrac{1}{48} + \tfrac{1}{16}\right) = \tfrac{35}{384}\lambda^4,
    \end{align*}
    so we obtain $A^V_{2}(\lambda) = \frac{35}{384}\lambda^4 + \frac{5}{32} \lambda^3 + \frac{19}{64}\lambda^2 + \frac{5}{32} \lambda + \frac{35}{384}$.
\end{example}

From now on, we restrict to the case where $V$ is \emph{homogeneous} of degree $k$ in the $\bb{x}$ variables. Consequently, we only enumerate $k$-regular graphs. This assumption lets us rewrite $A^V_n$ in terms of an exponential integral, which will be crucial in the forthcoming section. Note that if a $k$-regular graph $G$ has Euler characteristic $-n$, then, since $k|V_G| = 2|E_G|$ and $|V_G| - |E_G| = -n$, $G$ is required to have $\tfrac{2n}{k-2}$ many vertices and $\tfrac{nk}{k-2}$ many edges.

\begin{proposition}
  \label{prop:An_exp_integral}
  Let $V\in \QQ[\lambda][\bb{x}]$ be homogeneous of degree $k$ in $\bb{x}$, and let us denote $M = \tfrac{k}{k-2}$ and $K = \tfrac{2}{k-2}$. Then, for any integer $n\geq 0$ such that $nK,nM\in\ZZ$, we have
  \begin{equation}
    \label{eq:An_exp_integral}
    A^V_n (\lambda) = \frac{2^{nM + (d-2)/2} \left(nM +\frac{d-2}{2}\right)!}{(2\pi)^{d/2} (nK)!} \int_{S^{d-1}} V(\bb{x},\lambda)^{nK} \,\omega_{S^{d-1}},
  \end{equation}
  where $\omega_{S^{d-1}}$ denotes the standard volume form of the $(d-1)$-dimensional sphere $S^{d-1}$.
\end{proposition}

\begin{proof}
  Since $|V_G|$ and $|E_G|$ must be integers, $A^V_n = 0$ unless $nK,nM\in\ZZ$. So let us assume the latter. From Corollary \ref{cor:gen_fun} we have 
  \begin{equation}
    \label{eq:An_series}
    A^V_n(\lambda) = \sum_{\substack{\bb{s} \geq 0 \\ |\bb{s}| = nM}} (2s_1-1)!!\cdots (2s_d-1)!!\cdot [\bb{x}^{2\bb{s}}]\exp\left( \sum_{\bb{w}\in \ZZ^d_{\geq 0},\, |\bb{w}| \geq 1} \Lambda_{\bb{w}}(\lambda) \frac{\bb{x}^{\bb{w}}}{\bb{w}!} \right).
  \end{equation}
  The $\bb{x}^{2\bb{s}}$ coefficient of the exponential on the left-hand side has degree $\tfrac{2nM}{k} = nK$ in the coefficients $\lambda_{\bb{w}}$. Using the homogeneity of $V$, this implies
  \[
    [\bb{x}^{2\bb{s}}]\exp\left( \sum_{\bb{w}\in \ZZ^d_{\geq 0},\, |\bb{w}| \geq 1} \Lambda_{\bb{w}}(\lambda) \frac{\bb{x}^{\bb{w}}}{\bb{w}!} \right) = [\bb{x}^{2\bb{s}}] \frac{V(\bb{x},\lambda)^{nK}}{(nK)!}.
  \]
  Recall the Gaussian integral identity
  \[
    \frac{1}{\sqrt{2\pi}} \int_{-\infty}^{\infty} e^{-\frac{x^2}{2}} x^t \d x = \left\{ \begin{array}{ll}
      (t-1)!! & \text{if } t \text{ is even,} \\
      0 & \text{if } t \text{ is odd.}
    \end{array} \right.
  \]
  This enables us to rewrite \eqref{eq:An_series} as 
  \[
    A^V_n(\lambda) = \frac{1}{(2\pi)^{d/2} (nK)!} \int_{\RR^d} \exp\left( -\frac{1}{2} \sum_{i=1}^{d} x_i^2 \right) V(\bb{x},\lambda)^{nK} \d\bb{x}.
  \]
  Finally, the homogeneity of $V$ allows rescaling $\bb{x}$ to have unit norm via $V(r\cdot \bb{x},\lambda) = r^k V(\bb{x},\lambda)$, so we can rewrite the integral above in polar coordinates as
  \begin{align*}
    A^V_n(\lambda) & = \frac{1}{(2\pi)^{d/2} (nK)!} \int_{0}^{\infty} r^{d-1} \left( \int_{\partial B_r(0)} \exp\left( -\frac{1}{2} \sum_{i=1}^{d} x_i^2 \right) V(\bb{x},\lambda)^{nK} \,\omega_{\partial B_r(0)} \right) \d r\\
    & = \frac{1}{(2\pi)^{d/2} (nK)!} \int_{0}^{\infty} e^{-\frac{r^2}{2}} r^{d-1 + knK} \d r \int_{S^{d-1}} V(\bb{x},\lambda)^{nK} \,\omega_{S^{d-1}}.
  \end{align*}
  Together with the identity
  \[
    \int_{0}^{\infty} e^{-\frac{r^2}{2}} r^{nkK+d-1} \d r = \int_{0}^{\infty} e^{-q} (2q)^{(knK + d-2)/2} \d q = 2^{nM + (d-2)/2} \left(nM + \frac{d-2}{2}\right)!\,,
  \]
  this yields the desired expression.
\end{proof}

In the large-$n$ asymptotics, a saddle point approximation of \eqref{eq:An_exp_integral} can be used to derive an asymptotic expression for $A^V_n$ in terms of maxima of $|V(\bb{x},\lambda)|$ on $S^{d-1}$, as long as $\lambda$ is \emph{real}. A detailed analysis of this is carried out in \cite{multicolor}. However, since we are interested also in complex values of $\lambda$, we need to analytically continue the $n\rightarrow\infty$ limit of \eqref{eq:An_exp_integral}. This is the content of the next section, using Picard--Lefschetz theory.

\section{Lefschetz thimble decomposition}
\label{sec:Lefschetz}

Proposition \ref{prop:An_exp_integral} tells us to study the analytic continuation to $\lambda\in\CC$ of the integral
\begin{equation}
  \label{eq:int_sphere}
  I^V_n(\lambda) := \int_{S^{d-1}} V(\bb{x},\lambda)^n \omega_{S^{d-1}}
\end{equation}
in the limit $n\rightarrow \infty$. To this end, we decompose $I_n^V(\lambda)$ in terms of integrals over \emph{Lefschetz thimbles}. These are integration contours associated to critical points of $V$ along which the real part of $\log V$ decreases; this makes them amenable to a stationary phase approximation. In the large-$n$ limit, the thimble integrals reduce to critical point evaluations.\par 
Such an approach is prominent both in mathematics and theoretical physics, see \cite{delabaere2002global, witten2011analytic}. To make it precise we need to show that the Lefschetz thimbles define a basis of the appropriate homology group. Here, a technical subtlety arises: 
one commonly assumes the critical values $\Re \log V(\sigma_j)$ to be \emph{distinct} for different critical points $\sigma_j$, see e.g.\ \cite[Hypothesis H4]{delabaere2002global} or \cite[Assumption 8]{aomoto2011theory}. However, for the applications we have in mind, this assumption is violated, e.g.\ for $V$ as constructed in Example \ref{ex:ising_expansion}. \par 
We modify the construction of the Lefschetz thimbles in a two-step process which has been laid out in \cite[\S 4]{matsubara2023localization}: 
First, we define thimbles $\Gamma_\sigma$ as (un)stable manifolds of a Morse--Smale vector field that is a small perturbation of the gradient vector field used in the standard Lefschetz thimble construction. This ensures transverse intersections of the thimbles. 
Contours constructed in this way are elements in a \emph{locally finite} homology group. \par

In a second step, we need to \emph{regularise} these thimbles to obtain honest (twisted) homology classes $\widetilde{\Gamma}_\sigma$. This involves a real oriented blowup construction. For the case of complements of hyperplane arrangements, this process is also described in \cite[\S 3.2.3--3.2.5]{aomoto2011theory}.
Under a genericity assumption on $V$ expressing the number of critical points as a signed Euler characteristic, the $\widetilde{\Gamma}_\sigma$ form a basis of the twisted homology group.
Since the integral \eqref{eq:int_sphere} can be regarded as a period pairing between this twisted homology group and an algebraic de Rham cohomology group, this then proves the Lefschetz thimble decomposition of the integral \eqref{eq:int_sphere}, which is the content of Theorem \ref{thm:lefschetz_decomp}. By applying a stationary phase approximation, we can then get an asymptotic expression for $I_n^V(\lambda)$ that is valid for almost all $\lambda\in\CC$, see Corollary \ref{cor:stationary_phase}.\par

Let us introduce some notation, mostly following \cite{matsubara2023localization}. Let 
\[ 
  \Sc := \{ \bb{x}\in \CC^d \,:\, x_1^2 + x_2^2 + \dots + x_d^2 = 1\}
\]
be the \emph{complexified sphere} of complex dimension $d-1$. For a fixed $\lambda$, let $U_\lambda := \Sc \setminus \V(V(\bb{x},\lambda))$ be the locus where $V$ is non-vanishing, equipped with a complete Riemannian metric $g$. We can view $D_\lambda := \V(V(\bb{x},\lambda))$ as a divisor on $\Sc$. Moreover, we denote $F_\lambda(\bb{x}) := \log V(\bb{x},\lambda)$. In the following we collect the necessary non-degeneracy assumptions. Crucially, these only fail on an at most zero-dimensional discriminant locus $\Sigma$, see Remark \ref{rem:discriminant}.
\begin{itemize}
  \label{itm:assumptions}
  \item[(A1)] \phantomsection \label{ass:A1} The critical locus $\crit(F_\lambda) := \{\bb{x} \in U_\lambda \,:\, \d_{\bb{x}} F_\lambda(\bb{x}) = 0\}$ is a zero-dimensional variety. This is true as long as the factorisation of $V(\bb{x},\lambda)$ is squarefree.
  \item[(A2)] \phantomsection \label{ass:A2} The divisor $D_\lambda$ should be simple normal crossing. In particular, this is the case if $D_\lambda$ is smooth, i.e.\ $\d_{\bb{x}} V(\bb{x},\lambda) \neq 0$ for all $\bb{x} \in D_\lambda$. 
 \item[(A3)] \phantomsection \label{ass:A3} For any $\sigma\in \crit(F_\lambda)$, $F_\lambda$ is non-degenerate at $\sigma$, i.e.\ $\det \Hess_{F_\lambda}(\sigma) \neq 0$.
 \item[(A4)] \phantomsection \label{ass:A4} The number of critical points is $|\crit(F_\lambda)| = (-1)^{d-1} \chi(U_\lambda)$, where $\chi$ is the topological Euler characteristic. We denote this number by $\chi^\ast := |\crit(F_\lambda)|$.
\end{itemize}

Note that $\crit(F_\lambda)$ can easily be computed by first determining the critical points of $V$ on $\Sc$ via Lagrange multipliers and then intersecting those with $U_\lambda$.

\begin{remark}
  The Euler characteristic of $\Sc$ is $1 + (-1)^{d-1}$. Indeed, let $f=x_1^2 + \dots + x_d^2$; then the affine Milnor fibration $\CC^d\setminus \{f = 0\} \xrightarrow{f} \CC^*$ has fibre $\{f=1\}$ which is homotopy equivalent to a bouquet of $\mu$ (complex) spheres of dimension $d-1$. Here, $\mu$ is the Milnor number of the singularity of $f$ at the origin, which is an $A_1$-singularity, hence $\mu=1$. This implies $\chi(\Sc) = 1 + (-1)^{d-1}$. By the excision property of the Euler characteristic,
  \begin{equation}
    \label{eq:Euler_char_U}
    (-1)^{d-1}\chi(U_\lambda) = (-1)^d \left( |D_\lambda| - 1 \right) + 1.
  \end{equation}
\end{remark}

\begin{example}
  \label{ex:critical_point_count}
  Consider $V(x_1,x_2,\lambda) = \frac{x_1^4}{4!} + \lambda \frac{x_1^2 x_2^2}{2!\cdot 2!} + \lambda^2 \frac{x_2^4}{4!}$; generically, this polynomial has eight roots on $\Sc = \{x_1^2+x_2^2=1\}$ which has Euler characteristic zero. By \eqref{eq:Euler_char_U}, one expects $|\crit(F_\lambda)| = 8$. This is indeed the case for generic $\lambda$. For $\lambda\in\{1/3,3\}$, there are only six critical points. If $\lambda = 3\pm\sqrt{8}$, $V$ has only four roots on $\Sc$, as well as four critical points.
\end{example}

Let $\nabla_\pm := \d \pm \tau \dlog V \wedge$ be a \emph{twisted differential} on $\Sc$ with logarithmic poles along $D_\lambda$. The $(d-1)$-form $\omega_{S^{d-1}}$ from \eqref{eq:int_sphere} is an element of $H^0(U_\lambda, \Omega^{d-1}_{U_\lambda})$. 
We can also view it as a representative of a cohomology class in an algebraic de Rham cohomology group $H^{d-1}_\pm(\tau) := \mathbb{H}^{d-1}(U_\lambda;(\Omega^\bullet_{U_\lambda},\nabla_\pm))$. Let $\L^{\pm}(\tau)$ be the sheaf on $U_\lambda$ of flat sections of the connection $\nabla_\mp$. Then there is the \emph{period pairing}, which is the perfect pairing
\begin{align}
  \label{eq:period_pairing}
  \langle \bullet,\bullet \rangle_{\per} \,:\, H_{d-1}(U_\lambda;\L^\pm(\tau)) \otimes_\CC H^{d-1}_\pm(\tau) & \rightarrow \CC \\
  [\Gamma^\pm] \otimes [\omega_\pm] & \mapsto \int_{\Gamma^\pm} V(\bb{x},\lambda)^{\pm\tau} \omega_{\pm}. \nonumber
\end{align}
Here, $H_{d-1}(U_\lambda;\L^\pm(\tau))$ is a \emph{twisted homology group}, taking values in the local system $\L^\pm(\tau)$. The reader is referred to \cite[\S 3]{matsubara2023four} for an introduction to twisted (co)homology. 
In the following we omit the scaling $\tau$ to ease notation, and assume $\tau > 0$ to be \emph{generic}.

\subsection{Step 1: construction of Lefschetz thimbles as locally finite chains}

In this subsection we explain the construction of Lefschetz thimbles, giving rise to \emph{locally finite} homology classes $[\Gamma_\sigma^\pm] \in H_{d-1}^\lf(U_\lambda;\L^\pm)$. This construction is basically the classical gradient flow trajectory construction, except that one needs to perturb the gradient field to ensure transverse intersections of stable and unstable manifolds.\par   

Let $T_\lambda$ be an open neighbourhood of $D_\lambda$ with the properties that 
\begin{itemize}
  \item $\Sc \setminus T_\lambda$ and $\overline{T}_\lambda$ are manifolds with boundary;
  \item the critical locus is contained in the complement of $T$, $\crit(F_\lambda) \subset \Sc \setminus T_\lambda$;
  \item the gradient $\d_{\bb{x}}\Re F_\lambda$ is transversal to $\partial(\Sc\setminus T_\lambda)$ and does not vanish on $\partial T$.
\end{itemize}
Due to a theorem by Smale \cite[Theorem A]{smale1961gradient}, assuming \hyperref[ass:A1]{(A1)} and \hyperref[ass:A3]{(A3)}, there exists a $C^1$ \emph{Morse--Smale vector field} $\xi_\lambda$ on $\Sc \setminus T_\lambda$ adapted to $\Re F_\lambda$. We explain the terminology below, see \cite[\S 2.4]{MorseTheoryInvitation} for more details. The vector field $\xi_\lambda$ is \emph{gradient-like} with respect to $\Re F_\lambda$ if $\xi_\lambda \cdot \Re F_\lambda > 0$ on $(\Sc \setminus T_\lambda)\setminus \crit(F_\lambda)$, and for any critical point $\sigma\in \crit(F_\lambda)$ there exist local coordinates $(s_i)$ such that $\xi_\lambda = -2 \sum_{i=1}^m s_i \partial_{s_i} + 2 \sum_{i > m} s_i \partial_{s_i}$, where $m$ is the index of $F_\lambda$ at $\sigma$. 
Let $\Phi_t$ be the flow on $\Sc\setminus T_\lambda$ generated by $\xi$. The \emph{(un)stable manifold} of $\sigma$ relative to $\xi_\lambda$ is the set $\{ \bb{x} \in \Sc\setminus T_\lambda \,:\, \lim_{t\to \pm\infty}  \Phi_t(\bb{x}) = \sigma\}$. A gradient-like vector field $\xi_\lambda$ is \emph{Morse--Smale} if for any pair of critical points, the associated stable and unstable manifolds intersect transversely. One can think of $\xi_\lambda$ as a small perturbation of the gradient vector field defined by $\d_{\bb{x}}\Re F_\lambda$, where the perturbation ensures transverse intersections. \par   

By Whitney's extension theorem, it is possible to extend $\xi_\lambda$ to a $C^1$-vector field on $U_\lambda = \Sc \setminus D_\lambda$ so that $\xi_\lambda \cdot \Re F_\lambda > 0$ on $U_\lambda\setminus \crit(F_\lambda)$; by abuse of notation we call this extension again $\xi_\lambda$, and denote its flow by $\Phi_t$. For any critical point $\sigma\in \crit(F_\lambda)$, we can then define the Lefschetz thimbles
\[
  \Gamma^{\pm}_\sigma := \left\{ \bb{x} \in U_\lambda \,:\, \lim_{t\to\pm\infty} \Phi_t(\bb{x}) = \sigma \right\}.
\]
The transversality of stable and unstable manifolds implies $\Gamma^+_\sigma \cap \Gamma^-_\tau = \varnothing$ if $\sigma \neq \tau$, see \cite[Remark 2.4.9]{MorseTheoryInvitation}. Around a critical point $\sigma$, by the Morse lemma, we can find coordinates $(z_i) = (s_i + \i t_i)$ such that $F_\lambda(\bb{z}) = F_\lambda(\sigma) + \sum_{i=1}^{d-1} z_i^2$. This involves a choice of branch of $F_\lambda$ near $\sigma$, i.e.\ a choice of value of $\Im F_\lambda(\sigma)$. Then, for $\epsilon > 0$ small, we define
\begin{equation}
  \begin{aligned}
  \label{eq:def_c_eps}
  c_\epsilon^+ & := \left\{ \bb{z} \in U_\lambda \,:\, \sum_{i=1}^{d-1} s_i^2 = \epsilon,~ t_1=\dots =t_{d-1}=0 \right\}, \\
  c_\epsilon^- & := \left\{ \bb{z} \in U_\lambda \,:\, \sum_{i=1}^{d-1} t_i^2 = \epsilon,~ s_1=\dots =s_{d-1}=0 \right\}.
\end{aligned}
\end{equation} 
With these notions, we can rewrite the thimbles as 
\[
  \Gamma^{\mp}_\sigma = \bigcup_{\pm \{t\geq 0\}} \Phi_t(c_\epsilon^\pm) \cup \bigcup_{0 \leq \delta\leq \epsilon} c^\pm_\delta,
\]
confer e.g.\ \cite[\S 2.3.1]{delabaere2002global}. From this expression we see that $\Gamma^\pm_\sigma$ define classes in the locally finite twisted homology group $[\Gamma^\pm_\sigma] \in H^{\lf}_{d-1}(U_\lambda;\L^\pm)$. Its elements are classes of locally finite sums of twisted singular chains, i.e.\ formally infinite sums $\sum_{\Delta} c_\Delta \Delta\otimes U_\Delta$, where $\Delta$ is a singular simplex and $U_\Delta$ a branch of $F_\lambda$ on $\Delta$, such that for any compact subset $K\subset U_\lambda$, there are only finitely many $\Delta$ whose image intersect $K$ nontrivially and with $c_\Delta \neq 0$.

\subsection{Step 2: regularising the Lefschetz thimbles}

To obtain Lefschetz thimbles in the usual twisted homology group we need to \emph{regularise} the classes $[\Gamma^\pm_\sigma] \in H^{\lf}_{d-1}(U_\lambda;\L^\pm)$. This is done by constructing a real oriented blowup of $\Sc$ along $D_\lambda$, which, roughly speaking, replaces $D_\lambda$ with ``pipe-like'' objects. The regularisation of the thimble then carries the datum of a direction at $D_\lambda$. This construction is taken from the proof of \cite[Lemma 4.2]{matsubara2023localization} and the paragraph thereafter. See also \cite[\S 3.2.5]{aomoto2011theory}.\par

Let $\varpi\,:\, \widetilde{\Sc} \rightarrow \Sc$ be the \emph{real oriented blowup} of $\Sc$ along $D_\lambda$, see \eqref{eq:blowup} below and also \cite[\S 8.2]{sabbah2012introduction}. Under assumption \hyperref[ass:A2]{(A2)}, this lets $\widetilde{\Sc}$ carry the structure of a manifold with corners, i.e.\ local charts are open sets in $[0,\infty)^k \times \RR^{d-1-k}$ \cite[Definition 2.1]{joyce2009manifolds}. 
The usual definitions of smoothness and tangent spaces extend to this setting. Let $X$ be a manifold with corners. For a point $\bb{x}\in X$, we define the \emph{inward sector} $\IS(T_{\bb{x}}X)$ as follows \cite[Definition 2.2]{joyce2009manifolds}: let $(U, \varphi)$ be a chart on $X$ with $U$ an open subset of $[0,\infty)^k \times \RR^{d-1-k}$ containing zero such that $\varphi(0) = \bb{x}$. 
The map $\restr{\d\varphi}{0}\,:\, T_0U = \RR^{d-1} \rightarrow T_{\bb{x}}X$ is an isomorphism. Then $\IS(T_{\bb{x}}X)$ is the image $\IS(T_{\bb{x}}X) = \restr{\d\varphi}{0}([0,\infty)^k \times \RR^{d-1-k}) \subseteq T_{\bb{x}}X$. Visually speaking, it contains the tangent vectors ``pointing into'' $X$.\par 

For an irreducible component $D_i\subset D_\lambda$, we denote the pullback $\widetilde{D}_i := \varpi^{-1}(D_i)$. Let us also write $D_0$ and $D_\infty$ for the support of the divisors of zeros and poles of $e^{F_\lambda}$, respectively, and $D = D_0 \cup D' \cup D_\infty$. Let $(z_i)$ be local coordinates around a point in $D_\infty$. Then on $\widetilde{\Sc}$ we have corresponding local coordinates $((r_i, e^{\i \theta_i})_{i=1}^k, z_{k+1},\dots,z_{d-1})$ and the blowup map is locally given by (which can be regarded as the definition of a real oriented blowup)
\begin{equation}
  \label{eq:blowup}
  \varpi \,:\, ((r_i, e^{\i \theta_i})_{i=1}^k, z_{k+1},\dots,z_{d-1}) \mapsto ((r_i e^{\i \theta_i})_{i=1}^k, z_{k+1},\dots,z_{d-1}).
\end{equation}
We may assume that $\widetilde{D}_\infty$ and $\widetilde{D}_0$ are locally given by $\widetilde{D}_\infty = \{r_1\dots r_\ell = 0\}$ and $\widetilde{D}_0 = \{r_{\ell+1}\dots r_k=0\}$. 
Then we can define a vector field locally via $\Theta = \sum_{i=1}^{\ell} \partial_{r_i}$. Evaluated at a point $\bb{x}\in\Sc$, this gives an inward pointing vector $\Theta(\bb{x}) \in \IS(T_{\bb{x}}\widetilde{\Sc})$. 
Using a partition of unity, one can globally extend this to a vector field $\Theta\,:\, \widetilde{\Sc} \rightarrow \IS(T\widetilde{\Sc})$ which is tangent to every component of $D'$. Let $\Psi_t$ be the flow of $\Theta$. 
Then, for a fixed $T>0$, we define 
\[
  \widetilde{W}_\infty := \bigcup_{0\leq t\leq T} \Psi_t(\widetilde{D}_\infty) \quad\text{and}\quad W_\infty := \varpi(\widetilde{W}_\infty).
\]
We claim that we can find $M>\epsilon$ and $T>0$ large enough so that 
\begin{equation}
  \label{eq:inclusions}
  \Phi_T(c_\epsilon^+) \subseteq \{\Re F_\lambda > M\} \subseteq W_\infty.
\end{equation}
Since $\Theta$ is inward-pointing, $\Re F_\lambda$ decreases along the image of the flow $\Psi$ under $\varpi$. This shows the existence of $M$ so that the containment on the right-hand side holds. Moreover, we have $\lim_{t\to\infty} \Phi_t(\sigma) \subseteq D_\infty = \{\Re F_\lambda = +\infty\}$. Since $\lim_{n\to\infty} c^+_{1/n} = \{\sigma\}$, an application of \cite[Lemma 8.4.5]{jost2017Riemannian} shows that for any $T>0$, $\Phi_t(c^+_{1/n})|_{[-T,T]}$ converges to $\restr{\Phi_t(\sigma)}{[-T,T]}$. Therefore, for any $M$, there exists a $T$ such that the containment on the left-hand side of \eqref{eq:inclusions} holds, which concludes the proof of the claim. \par

By the vanishing theorem of twisted homology, see e.g.\ \cite[Theorem A.1]{agostini2022vector}, $H_{d-2}(W_\infty \setminus D; \L^-) = 0$. Therefore, there exists a chain $C\in C_{d-1}(W_\infty\setminus D;\L^-)$ with $\partial C = -\Phi_T(c_\epsilon^+)$. We then define the \emph{regularised Lefschetz thimble} as the singular chain
\[
  \widetilde{\Gamma}^{-}_\sigma := \bigcup_{0\leq t\leq T} \Phi_t(c_\epsilon^+) + \bigcup_{0 \leq \delta\leq \epsilon} c^+_\delta + C \in C_{d-1}(U_\lambda;\L^-).
\]
Analogously, we can define $\widetilde{\Gamma}^{+}_\sigma$, using $W_0$ instead of $W_\infty$. By \cite[Lemma 4.2]{matsubara2023localization}, the locally finite homology group $H_p^{\lf}(W_\infty \setminus D; \L^-) = 0$ vanishes for any $p$. In particular, this implies $[\Gamma^-_\sigma - \widetilde{\Gamma}^-_\sigma] = [\bigcup_{T < t} \Phi_t(c_\epsilon^+) - C] = 0 \in H_{d-1}^{\lf}(W_\infty \setminus D; \L^-)$. Therefore, there exists a locally finite chain $C'$ on $W_\infty \setminus D$, and hence on $U_\lambda$, with $\partial C' = \Gamma^-_\sigma - \widetilde{\Gamma}^-_\sigma$. \par

Moreover, $\partial \widetilde{\Gamma}^-_\sigma = \Phi_T(c_\epsilon^+) + 0 - \partial C = 0$ by construction of $C$, so $[\widetilde{\Gamma}^-_\sigma]$ gives an element of $H_{d-1}(U_\lambda;\L^-)$. In fact, $[\widetilde{\Gamma}^-_\sigma]$ is the image of $[\Gamma^-_\sigma]$ under the \emph{regularisation map}
\[
  \mathrm{reg} \,:\, H^{\lf}_{d-1}(U_\lambda;\L^\pm) \xrightarrow{\sim} H_{d-1}(U_\lambda;\L^\pm),
\]
which is the inverse of the natural isomorphism $H_{d-1}(U_\lambda;\L^\pm) \xrightarrow{\sim} H^{\lf}_{d-1}(U_\lambda;\L^\pm)$, under the assumption that $\tau$ is generic, confer \cite[\S 3.2.5]{aomoto2011theory}. Let us consider the (perfect) homology intersection pairing 
\[
  \langle \bullet,\bullet \rangle_{\mathrm{h}} \,:\, H_{d-1}(U_\lambda, \L^-) \otimes_{\CC} H_{d-1}(U_\lambda, \L^+) \rightarrow \CC.
\]
From the transversality property of the Morse--Smale vector field, we have
\[
  \langle [\widetilde{\Gamma}^-_\sigma], [\widetilde{\Gamma}^+_{\sigma'}]\rangle_{\mathrm{h}} = \delta_{\sigma{\sigma'}},
\]
so $\left\{ \Gamma^\pm_\sigma \right\}_{\sigma \in \crit(F_\lambda)}$ are independent elements of $H_{d-1}(U_\lambda;\L^\pm)$. By assumption \hyperref[ass:A4]{(A4)}, there are precisely $\chi^*$ many critical points. This coincides with the dimension of $H_{d-1}(U_\lambda;\L^\pm)$, see \cite[Theorem A.1]{agostini2022vector}. Therefore, we arrive at the following statement.
\begin{theorem}
  \label{thm:lefschetz_decomp}
  For a fixed $\lambda\in\CC$ so that assumptions \hyperref[ass:A1]{(A1)--(A4)} are valid, and for generic $\tau>0$, the regularised Lefschetz thimbles $\left\{ \Gamma^\pm_\sigma(\tau) \right\}_{\sigma \in \crit(F_\lambda)}$ form a basis of the twisted homology $H_{d-1}(U_\lambda;\L^\pm(\tau))$. Via the period pairing \eqref{eq:period_pairing}, the integral $I_\tau^V(\lambda)$ decomposes as 
  \begin{equation}
    \label{eq:thm_Lefschetz_decomp}
    I_\tau^V(\lambda) = \sum_{\sigma\in\crit(F_\lambda)} c_\sigma(\lambda) \int_{\widetilde{\Gamma}^+_\sigma(\tau)} V(\bb{x},\lambda)^\tau \omega_{S^{d-1}},
  \end{equation}
  for some coefficients $c_\sigma(\lambda) \in \CC$.
\end{theorem}

\begin{corollary}
  \label{cor:stationary_phase}
  Let us fix $\lambda\in\CC$ so that assumptions \hyperref[ass:A1]{(A1)--(A4)} are valid. Then, in the large-$n$ limit, the integral $I_n^V(\lambda)$ is asymptotically given by
  \begin{equation}
    \label{eq:asymptotic_thimble_expansion}
    I_n^V(\lambda) \sim \left( -\frac{2\pi}{n} \right)^{(d-1)/2} \sum_{\sigma\in\crit(F_\lambda)} \frac{c_\sigma(\lambda)\omega_{S^{d-1}}(\sigma\partial_{\bb{x}})}{\sqrt{\det \Hess_{F_\lambda}(\sigma)}} e^{n F_\lambda(\sigma)} (1+ o(n^{-1})),
  \end{equation}
  where $\sigma\partial_{\bb{x}}$ denotes the constant vector field on $\Sc$ with value $\sigma$.
\end{corollary}

\begin{proof}
  For generic $n>0$, we obtain the Lefschetz thimble decomposition of $I^V_n(\lambda)$ from Theorem \ref{thm:lefschetz_decomp}. 
  In each of the integrals in \eqref{eq:thm_Lefschetz_decomp} we then change the integration contour $\widetilde{\Gamma}^+_\sigma(n)$ to the unregularised thimble $\Gamma^+_\sigma(n)$. The contribution 
  \[
    \int_{\bigcup_{0\leq t\leq T} \Phi_t(c_\epsilon^+) + C} V(\bb{x},\lambda)^n \omega_{S^{d-1}}
  \]
  is exponentially suppressed in the large-$n$ limit, see \cite[(4.5)]{matsubara2023localization}.
  Therefore, we can apply the stationary phase formula (see \cite[(4.7)]{matsubara2023localization}) to $\Gamma^+_\sigma(n)$ which is defined for any $n>0$, to obtain an asymptotic expansion for $I_n^V(\lambda)$. This gives the desired expression \eqref{eq:asymptotic_thimble_expansion}.
\end{proof}

We conclude this section with some remarks. 

\begin{remark}
  \label{rem:discriminant}
  If $V(\bb{x},\lambda)\in \CC[\bb{x},\lambda]$ is irreducible, then assumptions \hyperref[ass:A1]{(A1)--(A4)} are violated on an at most zero-dimensional variety in $\CC_\lambda$. Let us denote this finite \emph{discriminant} set of points by $\Sigma$. In case of \hyperref[ass:A1]{(A1)}, the claim follows from Sard's theorem. The divisor $D_\lambda$ is smooth away from a discriminant hypersurface in the one-dimensional parameter space $\CC_\lambda$. Similarly, the vanishing Hessian is a hypersurface in $\CC_\lambda$, giving a finite set of exceptional points. The locus where \hyperref[ass:A4]{(A4)} fails is a version of an \emph{Euler discriminant}, see e.g.\ \cite{telen2024euler}.
\end{remark}

\begin{remark}
  \label{rem:Stokes}
  By construction of $\widetilde{\Gamma}^+_\sigma$, each of the integrals appearing in \eqref{eq:thm_Lefschetz_decomp} is an analytic function in $\lambda$ on $\CC\setminus\Sigma$, as well as the function $I^V_n(\lambda)$. That is, however, not true for the coefficients $c_\sigma(\lambda)$, which is an instance of \emph{Stokes' phenomenon}. These coefficients can even jump discontinuously along so-called \emph{Stokes curves} of the form 
  \[
    S_{\sigma,\sigma'} := \left\{ \lambda \in \CC \,:\, \Im F_\lambda(\sigma) = \Im F_\lambda(\sigma')\mod 2\pi \right\},
  \] 
  for $\sigma, \sigma'\in \crit(F_{\lambda})$ such that $S_{\sigma,\sigma'}$ is one-dimensional if interpreted as a real algebraic variety in $\RR^2 \simeq \CC$. If $\lambda$ lies on such a Stokes curve, the choice of basis of Lefschetz thimbles is not canonical any more, as can be seen from \eqref{eq:def_c_eps}. As $\lambda$ crosses a Stokes line, the basis of Lefschetz thimbles transforms via a unipotent matrix, see e.g.\ \cite{pham1985descente}. Note, however, that the $c_\sigma(\lambda)$ possess well-defined one-sided limits to the Stokes curves.\par

  For $\sigma,\sigma'\in\crit(F_\lambda)$ such that $\dim_{\RR}(S_{\sigma,\sigma'}) = 1$, we denote by 
  \[
    \cS^c_{\sigma,\sigma'} := \pi_0 \left( \RR^2 \setminus S_{\sigma,\sigma'} \right)
  \]
  the connected components in the complement of the Stokes curve. In particular, we have that $c_\sigma(\lambda) \neq 0$ on a union of regions in $\cS^c_{\sigma,\sigma'}$, except for possibly isolated points.
\end{remark}

\section{Accumulation of zeros}
\label{sec:accumulation}

In this section we state and prove the main theorem of this paper. The key idea is to combine Corollary \ref{cor:stationary_phase} with a result by Sokal on accumulation of zeros which is in turn a generalisation of the Beraha--Kahane--Weiss theorem \cite{beraha1978limits}. \par

Let us first set up some notation following \cite{sokal2004chromatic}. Let $D\subset \CC$ be a domain and $(f_n)_n$ a sequence of function $f_n\,:\, D\rightarrow \CC$. We define the following limit sets of zeros.
\begin{align*}
  \lim\inf \V(f_n) & := \left\{ \begin{array}{c}
    z\in D \,:\, \text{every neighbourhood } U \ni z \text{ has a nonempty} \\
    \text{intersection with all but finitely many } \V(f_n)
  \end{array} \right\} \\
  \lim\sup \V(f_n) & := \left\{ \begin{array}{c}
    z\in D \,:\, \text{every neighbourhood } U \ni z \text{ has a nonempty} \\
    \text{intersection with infinitely many } \V(f_n)
  \end{array} \right\}
\end{align*}
If $f_n$ is of the form $f_n(z) = \sum_{k=1}^m \alpha_k(z) \beta_k(z)^n$, we call an index $k$ \emph{dominant at} $z$ if $|\beta_k(z)| \geq |\beta_l(z)| ~ \forall 1\leq l\leq m$, and write $D_k := \left\{ z\in D \,:\, k\text{ is dominant at } z \right\}$. We call the following assumption the \emph{no-degenerate-dominance condition:}
\begin{enumerate}
  \item[(B1)] \phantomsection \label{ass:B1} There do not exist indices $k\neq k'$ such that $\beta_k \equiv \zeta\beta_{k'}$ for some $\zeta\in\CC \setminus \{1\}$ with $|\zeta| = 1$ and such that $D_k$ has nonempty interior.
\end{enumerate}
We are now in the position to formulate our main theorem precisely.

\begin{theorem}[Main result]
  \label{thm:main}
  Let $V \in \CC[\lambda][\bb{x}]$ be a polynomial satisfying assumptions \hyperref[ass:A1]{(A1)--(A4)} for generic $\lambda\in\CC$ (i.e.\ $\lambda\in\CC\setminus \Sigma$ with finite $\Sigma$). Moreover, assume that \hyperref[ass:B1]{(B1)} is satisfied for the functions $V(\sigma)$, $\sigma\in\crit(F_{\lambda})$. Then, for $n\to\infty$, the zeros of $A_n^V(\lambda)$ accumulate along segments of anti-Stokes curves. More precisely,
  \begin{equation}
    \label{eq:accumulation}
    \lim\inf \V(A_n^V) = \lim\sup \V(A_n^V) = \cP\, \cup \bigcup_{\substack{\sigma \neq \sigma' \in \crit(F_\lambda) \\ D \in \cS^c_{\sigma,\sigma'}}} \left\{ \lambda\in \CC \,:\, \Re F_\lambda(\sigma) = \Re F_\lambda(\sigma') \right\} \cap \overline{D},
  \end{equation}
  where $\cP$ is a (possibly empty) set of finitely many isolated points and the union ranges over some (not necessarily all) pairs of critical points $\sigma,\sigma' \in \crit(F_\lambda)$ such that the anti-Stokes curve 
  $\left\{ \lambda\in \CC \,:\, \Re F_\lambda(\sigma) = \Re F_\lambda(\sigma') \right\}$ is one-dimensional as a real algebraic variety.
\end{theorem}

\begin{remark}
  Note that the critical points of $F_\lambda = \log V$ agree with the critical points of $V$ restricted to $\Sc$. Moreover, we have the identities $\Re F_\lambda = \log |V(\lambda)|$ and $\Im F_\lambda = \arg V(\lambda) = \mathrm{atan}2(b,a)$ (the 2-argument arctangent function), from which one sees that the (anti-)Stokes curves are algebraic curves in $\RR^2_{ab}$ where $\lambda = a + \mathrm{i} b$. Therefore, the accumulation set of zeros in \eqref{eq:accumulation} is semialgebraic in the two real variables $a$ and $b$.
\end{remark}

Before we prove Theorem \ref{thm:main}, we illustrate the result with an example.

\begin{example}
  \label{ex:big_illustration}
  Consider $V(x_1,x_2,\lambda) = \frac{x_1^4}{4!} + \lambda \frac{x_1^2x_2^2}{2!\cdot 2!} + \lambda^2 \frac{x_2^4}{4!}$ as in Example \ref{ex:ising_expansion}. Using Lagrange multipliers the critical points of $V$ restricted to $\Sc$ can be computed as
  \begin{equation*}
    \label{eq:critical_points}
      \sigma_{1,2} = (\pm 1, 0) \quad \sigma_{3,4} = (0, \pm 1) \quad \sigma_{5,\dots,8} = \left( \pm \sqrt{\frac{\lambda^2 - 3\lambda}{\lambda^2 - 6\lambda + 1}}, \pm \sqrt{\frac{-3\lambda + 1}{\lambda^2 - 6\lambda + 1}} \right).
  \end{equation*}
  For generic $\lambda$, these are eight distinct critical points, which agrees with $|\chi(U_\lambda)|$, see Example \ref{ex:critical_point_count}. If $\lambda = 3\pm\sqrt{8}$, $\sigma_{5,\dots,8}$ move off to infinity. 
  However, in this case $V$ has only four roots on $\Sc$, so assumption \hyperref[ass:A4]{(A4)} is still satisfied. If $\lambda\in \{1/3, 3\}$, $\sigma_5$ and $\sigma_5$, respectively $\sigma_7$ and $\sigma_8$, collide and \hyperref[ass:A4]{(A4)} is violated. 
  These two points lie in the discriminantal locus $\{1/3, 3\} \subseteq \Sigma$, see Remark \ref{rem:discriminant}. The four critical points $\sigma_{1,\dots,4}$ are non-degenerate provided $\lambda\neq 0$. Moreover, $\sigma_{5,\dots,8}$ are non-degenerate as long as $\lambda\notin\{0,1\}$. No critical point lies on the divisor $D_\lambda$ except if $\lambda=0$. Therefore, we have $\Sigma = \{0,1/3,1,3\}$ as the finite discriminant locus where assumptions \hyperref[ass:A1]{(A1)--(A4)} are violated.\par

  There are three distinct anti-Stokes curves:
  \[
    C_1 = \{\Re F_{\lambda}(\sigma_1) = \Re F_{\lambda}(\sigma_5)\},~ C_2 = \{\Re F_{\lambda}(\sigma_3) = \Re F_{\lambda}(\sigma_5)\},~ C_3 = \{\Re F_{\lambda}(\sigma_1) = \Re F_{\lambda}(\sigma_3)\}.
  \]
  The curve $C_2$ is the purple curve depicted in Figure \ref{fig:Stokes_and_anti_Stokes}. It is a classical algebraic curve known as the \emph{Bernoulli lemniscate}, following the equation (with $a = \Re \lambda,\, b = \Im\lambda$)
  \[
    \ell(a,b) := ((a-3)^2 + b^2)^2 - 4^2((a-3)^2 - b^2) = 0.
  \]
  The curve $C_1$ is the green curve from Figure \ref{fig:Stokes_and_anti_Stokes}. On first sight it looks like a Pascal lima{\c c}on (such a lima{\c c}on has been found in \cite[\S 6]{dolan2001thin} as the accumulation locus for the zeros of a Potts model on a random graph, and in \cite{beaton2018partition} as the limit curve for the zeros of a generating function of adsorbing Dyck paths). However, that is not quite true for our curve $C_1$. Instead, it obeys the slightly modified equation $\ell(a,b) - 64((a^2 + b^2)^2 - 1) = 0$.\par

  The third anti-Stokes curve $C_3$ is a unit circle. It does not appear in the accumulation locus \eqref{eq:accumulation}, see also Figure \ref{fig:roots}. Finally, the Stokes curve arrangement needs to be determined. \par

  The locus $\{\Im F_{\lambda}(\sigma_1) = \Im F_{\lambda}(\sigma_5)\}$ can be identified as the union of the $a$-axis together with the ellipse defined by $3(a^2 + b^2) - a = 0$ which becomes apparent after rewriting as
  \[
  \arctan\left( \frac{2 b (a - 3 a^2 - 3 b^2)}{a^2 (1 + (a - 6) a) + (2 (a - 3) a - 1) b^2 + b^4} \right) = 0.
  \]
  Similarly, one identifies the Stokes curve $S_{\sigma_3,\sigma_5} = \{\Im F_{\lambda}(\sigma_3) = \Im F_{\lambda}(\sigma_5)\}$ as the union of the three lines $\{a=0\},\,\{a=3\},\,\{b=0\}$. Together, the Stokes arrangement is shown as the grey lines in Figure \ref{fig:Stokes_and_anti_Stokes}. The accumulation of zeros only takes place in the region
  \[
    \overline{D} = \left\{ (a,b)\in\RR^2 \,:\, a \leq 3 \text{ and }  3(a^2 + b^2) - a \geq 0 \right\}.
  \]
  In the complement of this region (grey shaded area in Figure \ref{fig:Stokes_and_anti_Stokes}), the corresponding Lefschetz thimbles do not contribute to the asymptotics \eqref{eq:asymptotic_thimble_expansion}. Note that the two points $\lambda = 1/3$ and $\lambda = 3$ lie in the (Euclidean) closure of $(C_1\cap \overline{D})\setminus \Sigma$, respectively $(C_2\cap \overline{D})\setminus \Sigma$, and hence belong to the accumulation locus since this is closed. 
  Our theory does not allow making statements regarding the two isolated points $\lambda \in \{0,1\}$, a subset of which forms $\cP$.
\end{example}

In order to prove Theorem \ref{thm:main}, we use Sokal's result \cite[Theorem 1.5]{sokal2004chromatic} below.

\begin{theorem}[Generalised Beraha--Kahane--Weiss]
  \label{thm:sokal}
  Let $D$ be a domain in $\CC$, and let $\alpha_1,\dots,\alpha_{m}$ and $\beta_1,\dots,\beta_{m}$ with $m\geq 2$ be \emph{analytic} functions on $D$, none of which is identically zero. For each integer $n\geq 0$, let 
  \[
    f_n(z) = \sum_{k=1}^{m} \alpha_k(z)\beta_k(z)^n.
  \]
  Let us assume that \hyperref[ass:B1]{(B1)} holds. Then $\lim\inf \V(f_n) = \lim\sup \V(f_n)$, and a point $z$ lies in this set if and only if either
  \begin{enumerate}
    \item there is a unique dominant index $k$ at $z$ and $\alpha_k(z) = 0$; or
    \item there are two or more dominant indices at $z$.
  \end{enumerate} 
\end{theorem}

Note that the locus where {\it 1.}~holds consists of isolated points, whereas {\it 2.}~leads to curves in $\RR^2\simeq \CC$. A small subtlety arises when one tries to apply Theorem \ref{thm:sokal} to the expression \eqref{eq:asymptotic_thimble_expansion}: the coefficients $c_\sigma(\lambda)$ are potentially non-analytic at Stokes curves, see Remark \ref{rem:Stokes}. This is the main point that needs to be addressed in the proof below.

\begin{proof}[Proof of Theorem \ref{thm:main}]
  From Proposition \ref{prop:An_exp_integral} we see that in the $n\to\infty$ limit, all zeros of $A_n^V(\lambda)$ come from zeros of the integral $I_n^V(\lambda) = \int_{S^{d-1}} V(\bb{x},\lambda)^{n} \,\omega_{S^{d-1}}$. 
  By Corollary \ref{cor:stationary_phase}, in the large-$n$ limit and under assumptions \hyperref[ass:A1]{(A1)--(A4)}, this integral asymptotically becomes
  \begin{equation}
    \label{eq:sum_proof}
    I_n^V(\lambda) \sim \sum_{\sigma \in \crit(F_\lambda)} \tilde{c}_{\sigma}(\lambda) e^{nF_{\lambda}(\sigma)} =: \tilde{I}_n^V(\lambda),
  \end{equation}
  where we absorbed all coefficients in $\tilde{c}_{\sigma}(\lambda)$. By assumption, this expression is valid for $\lambda\in \CC\setminus \Sigma$. The (isolated) points in $\Sigma$ potentially lie in $\cP$, so we now assume $\lambda\in \CC\setminus \Sigma$. The coefficient $c_\sigma$ is not identically zero on certain regions $D \in \cS^c_{\sigma,\sigma'}$, see Remark \ref{rem:Stokes}. If there are several critical points $\sigma,\sigma'$ with the same evaluations $F_\lambda(\sigma) = F_\lambda(\sigma')$, we collect the coefficients into one $\alpha_k(\lambda),~k=1,\dots,m$. 
  The result now follows from Theorem \ref{thm:sokal} except for the potential non-analyticity of $\alpha_k$ along Stokes curves.\par
  The only thing left to show is that if $\lambda_0\in S_{\sigma,\sigma'} \setminus \bigcup \{\text{anti-Stokes curves}\}$, then $\lambda_0 \notin \lim\sup\V(A_n^V)$ or $\lambda_0$ is isolated. Since $\lambda_0$ is not on any anti-Stokes curve, w.l.o.g.\ let $\sigma$ be the unique dominant critical point at $\lambda_0$. 
  Let $\epsilon > 0$ be small such that $\overline{B}_\epsilon := \{\lambda\in\CC \,:\, |\lambda - \lambda_0| \leq \epsilon\}$ is contained in a single component of the anti-Stokes curve arrangement complement and does not intersect any other Stokes curve. 
  There are several cases to consider. 
  If $c_\sigma$ and $c_{\sigma'}$ are both identically zero on $\overline{B}_\epsilon$, we can disregard $\sigma$ and $\sigma'$ from \eqref{eq:sum_proof} and may assume that all the remaining coefficients are analytic on $\overline{B}_\epsilon$, so the claim follows readily from Theorem \ref{thm:sokal}.
  If $\lambda_0$ is an isolated zero of both $c_\sigma$ and $c_{\sigma'}$, then $\lambda_0$ lies potentially in $\cP$ and we disregard this case. \par
  Finally, assume that $c_\sigma$ is nonzero in the intersection of $\overline{B}_\epsilon$ with a \emph{one-sided} neighbourhood of $\lambda_0$, i.e.\ there exists $D\in \cS^c_{\sigma,\sigma'}$ such that $c_\sigma$ is nonzero on $\overline{B}_\epsilon \cap \overline{D}$. 
  Note that even though $c_\sigma$ might not be analytic at $\lambda_0$ it has well-defined one-sided limits. We now adapt Sokal's argument in his proof of Theorem \ref{thm:sokal}. There exist constants $\delta > 0$ and $M < \infty$ such that for all $\lambda\in \overline{B}_\epsilon \cap \overline{D}$ we have $|\tilde{c}_\sigma(\lambda)| \geq \delta,~ | \tilde{c}_{\sigma'}(\lambda) / \tilde{c}_{\sigma}(\lambda) | \leq M$ for all $\sigma'\neq \sigma$, $|e^{F_\lambda(\sigma)}| \geq \delta$ and $| e^{F_\lambda(\sigma')} / e^{F_\lambda(\sigma)} | \leq 1-\delta$ for all $\sigma'\neq \sigma$. Then we have
  \[
    |\tilde{I}_n^V(\lambda)| \geq \delta^{n+1}\left( 1 - (m-1)M(1-\delta)^n \right)
  \]
  for all $\lambda\in \overline{B}_\epsilon \cap \overline{D}$. Therefore, $\tilde{I}_n^V(\lambda)$ is non-vanishing on $\overline{B}_\epsilon \cap \overline{D}$ for sufficiently large $n$ and hence so is $I_n^V(\lambda)$. Thus, $\lambda \notin \lim\sup \V(A_n^V)$, which concludes the proof.
\end{proof}

We conclude this section with a remark regarding the practicality of Theorem \ref{thm:main}.

\begin{remark}
  \label{rem:practicalities}
  Even though Theorem \ref{thm:main} is not explicit about which anti-Stokes curves and what parts of them appear as the limit objects, in practice one can decide this computationally via at least two different approaches.\par
  The first is to determine the Stokes structure in detail, i.e.\ determining which thimbles actually contribute to the sum \eqref{eq:sum_proof}. The integral $I_n^V$ is a holonomic function. Knowing the Lefschetz thimble decomposition at a single point (for example from the asymptotics for real $\lambda$ \cite{multicolor}), one can then numerically (with high precision) make an analytic continuation to a complex $\lambda$ using, for example, the software by Mezzarobba \cite{mezzarobba2016rigorous}.\par
  The second approach is to explicitly compute the zeros of $A_n^V$ for some large $n$ and observe which anti-Stokes curves are in proximity to the observed zeros. If one makes a careful analysis of the constants appearing in the proof of Theorem \ref{thm:main} and in Corollary \ref{cor:stationary_phase} and one chooses $n$ sufficiently large, this heuristic can be turned into a rigorous proof.
\end{remark}

\section{Connections to the Ising model}
\label{sec:Ising}

In this last section we establish connections to the Ising model and show how the edge-coloured graph counting framework unifies the theories of Lee--Yang and of Fisher zeros.
The main focus of this section is on the polynomial 
\begin{equation}
  \label{eq:master_polynomial}
  W(x_1,x_2,\mathcal{J},\mathcal{h}) := \sum_{i = 0}^{k} \mathcal{J}^i \mathcal{h}^{i \bmod 2}\frac{x_1^{i}x_2^{k-i}}{i!\cdot k!}.
\end{equation}
This polynomial specialises to two important instances:
\[
  V_{\rm Ising}^{(1)}(x_1,x_2,\lambda) := \restr{W(x_1,x_2,\J,\h)}{\J = \sqrt{\lambda},\h=0} ~, \quad V_{\rm Ising}^{(2)}(x_1,x_2,\lambda) := \restr{W(x_1,x_2,\J,\h)}{\J = 1,\h=\lambda}.
\]
For $k=4$, $V_{\rm Ising}^{(1)}(x_1,x_2,\lambda)$ is the polynomial from Example \ref{ex:ising_expansion}. We argue that for $V_{\text{Ising}}$, the polynomial $A_n^{V_{\text{Ising}}}(\lambda)$ can be interpreted as an \emph{average Ising model partition function} on a $k$-regular graph. The parameter $\J$ is a temperature-like parameter, while $\h$ corresponds to an external magnetic field parameter. The main theorem applied to $V_{\rm Ising}^{(2)}$ recovers a version of the original Lee--Yang theorem which we recall below. The main result applied to $V_{\rm Ising}^{(1)}$ can be interpreted as a result on the accumulation of Fisher zeros of a random regular graph. We explain these connections in the following. \par

Let $G$ be a (monocoloured) graph with vertices $V_G = \{ 1,\dots,|V_G| \}$. We denote an edge between vertices $i$ and $j$ by a multiset $\{i,j\} \in E_G$. The Ising model \cite{lenz1920Ising} on $G$ models the properties of a magnet depending on temperature, local interactions and an external magnetic field. To each vertex $i$ one associates a magnetic spin $\sigma_i \in \{\pm 1\}$. The Hamiltonian $H$ of the model accounts for nearest neighbour interactions, as well as external magnetic field parameters at each vertex,
\[
  H(\sigma) = -\sum_{\{i,j\}\in E_G} J_{ij}\sigma_i\sigma_j - \sum_{i \in V_G} h_i \sigma_i.
\]
Here, $J_{ij} \geq 0$ are ferromagnetic coupling parameters and $h_i$ are fugacities determining the external magnetic fields. Then the canonical \emph{partition function} of the Ising model is 
\[
  Z_G(h_1,\dots, h_{|V_G|},\beta; J_{ij}) = \sum_{\sigma \in \{\pm 1\}^{|V_G|}} e^{-\beta H(\sigma)},
\]
where $\beta = (k_B T)^{-1}$ is the inverse thermodynamic temperature ($k_B$ is the Boltzmann constant). Before we draw the connection to $A_n^{V_{\rm Ising}}$, let us recall the classic Lee--Yang theorem \cite{leeYang}. We state it in the form as presented in \cite{borcea2009lee}, setting $\beta = 1$.

\begin{theorem}[Lee--Yang]
  Let $J_{ij} \geq 0$. The partition function $Z_G(h_1,\dots,h_{|V_G|}, 1; J_{ij})$ does not vanish whenever $\Re(h_i) > 0$ for all $i=1,\dots,|V_G|$. In particular, all zeros of the univariate polynomial $Z_G(h,\dots,h, 1; J_{ij})$ lie on the imaginary axis.
\end{theorem}

We now make the common simplification to assume all fugacities $h_i$ to be equal to $h$. A convenient way to deal with the fugacity is to introduce \emph{Griffiths' ghost vertex} $\g\notin V_G$ with a fixed spin $\sigma_{\g} = +1$ and to define the interaction $J_{i\g} := h$ \cite{griffiths1967correlations}. Then the original Ising model can be interpreted as an Ising model without external magnetic field on an augmented graph $G'$ with vertices $V_{G'} = V_G \cup \{\g\}$ and edges $E_{G'} = E_G \cup \{ \{i,\g \} \,:\, i\in V_G \}$.\par

The connection to edge-coloured graph enumeration becomes apparent in the van der Waerden \emph{high energy expansion}, see e.g.\ \cite[\S 2.2.1]{duminil2016random}.
It expresses $Z_G(\beta, J)$ as 
\begin{equation}
  \label{eq:van_der_Waerden}
  Z_G(h,\beta;J_{ij}) = 2^{|V_G|} \left(\prod_{\{i,j\}\in E'}\cosh(\beta J_{ij})\right) \sum_{\substack{\gamma' \subseteq G' \\ \gamma' \setminus \g \text{ Eulerian}}} \prod_{\{i,j\}\in E_{\gamma'}} \tanh(\beta J_{ij}).
\end{equation}
The sum on the right-hand side of \eqref{eq:van_der_Waerden} ranges over all ``almost-Eulerian'' subgraphs of $G'$, by which we mean graphs all of whose vertices except $\g$ have even degree. 
Let $\gamma'\subseteq G'$ be such a subgraph and let $v\in V_{\gamma'} \setminus \{\g\}$ be a vertex. Then either $v$ is incident in $\gamma'$ to an even number of half-edges that belong to the original graph $G$, or $v$ is incident in $\gamma'$ to an odd number of half-edges of $G$ and is connected with $\g$. 
Therefore, the sum in \eqref{eq:van_der_Waerden} can be taken over all subgraphs $\gamma\subseteq G$, where a vertex comes with an additional contribution of $\tanh(\beta h)$ if it is incident to an odd number of half-edges in $\gamma$. \par

Given a $k$-regular graph $G$ and a subgraph $\gamma$ of $G$, we can define an edge-colouring on $G$ with two colours, depending on whether an edge belongs to $\gamma$. Moreover, the contributions $\tanh(\beta J_{ij})$ can be interpreted as vertex degree markings. Let us assume all interaction parameters $J_{ij}$ to be equal to a single parameter $J$ for all $i,j\in V_G$. Setting $\J = \sqrt{\tanh(\beta J)}$ and $\h = \tanh(\beta h)$, a vertex that is incident to $i$ half-edges of the colour indicating $\gamma$ contributes $\J^i \h^{i \bmod 2}$ to the sum in \eqref{eq:van_der_Waerden}. Therefore, $Z_G$ simplifies to 
\[
  Z_G(h,\beta; J) = (2h)^{|V_G|} \cosh(\beta J)^{|E_G|} \sum_{\gamma\subseteq G} \prod_{v\in V_G} \J^{\deg_\gamma(v)}\h^{\deg_\gamma(v) \bmod 2},
\]
where $\deg_\gamma(v)$ denotes the number of half-edges in $\gamma$ incident to $v$. Note that the prefactors $(2h)^{|V_G|} \cosh(\beta J)^{|E_G|}$ could also be absorbed into vertex degree markings, but we ignore this since it does not change the zeros of $A_n^W$. With this we obtain the following interpretation of $A_n^{W}$ as an average Ising model partition function.

\begin{theorem}
  Let $W = W(x_1,x_2,\J,\h)$ be as in \eqref{eq:master_polynomial}. Then the polynomial $A_n^{W}$ satisfies
  \[
    A_n^{W}(\h,\J) = \sum_{\substack{G \text{ $k$-regular} \\ \chi(G) = -n}} \frac{1}{|\Aut(G)|} \sum_{\gamma\subseteq G} \prod_{v\in V_G} \J^{\deg_\gamma(v)}\h^{\deg_\gamma(v) \bmod 2}.
  \]
  For the purpose of the automorphism group, $G$ is considered monocoloured.
\end{theorem}

\begin{proof}
  The only thing left to show is that if there are exactly $m$ subgraphs $\gamma_1,\dots,\gamma_m \subseteq G$ giving rise to the same isomorphism class of edge-bicoloured graphs $[\Gamma_1]=\dots =[\Gamma_m]$, then $|\Aut(\Gamma_i)| = |\Aut(G)|/m$. 
  Let $G$ have $2w$ many half-edges, while $\Gamma_i$ (for any $i$) has $2w_j$ half-edges of the $j^{\text{th}}$ colour. Note that $w = w_1+ w_2$. The group $\Sym_{2w}$ naturally acts on the half-edge labels of $\Gamma$, but also on the disjoint union $H_1\sqcup H_2$ of half-edge labels of $\Gamma_i$. In particular, we obtain an action of $\Aut(G)$ on the set of edge-coloured half-edge labelled graphs, and $\Gamma_1,\dots,\Gamma_m$ form an orbit of this action. Since the stabiliser is $\Aut(\Gamma_i)$, the claim follows from the orbit-stabiliser theorem.
\end{proof}

To apply our main result, we need to specialise the two-parameter polynomial $A_n^W(\h,\J)$ to a single parameter. There are two natural ways to do that. The first is as follows: we set the fugacity $\h=0$ to zero and consider zeros in the parameter $\J$. For convenience, we also make the substitution $\J^2 = \lambda$ to get the polynomial
\[
  V_{\rm Ising}^{(1)}(x_1,x_2,\lambda) = \restr{W(x_1,x_2,\J,\h)}{\J = \sqrt{\lambda},\h=0}.
\]
With this we obtain the following expression 
\[
  A_n^{V_{\rm Ising}^{(1)}}(\lambda) = \sum_{\substack{G \text{ $k$-regular} \\ \chi(G) = -n}} \frac{1}{|\Aut(G)|} \sum_{\substack{\gamma\subseteq G \\ \gamma \text{ Eulerian}}} \lambda^{|E_\gamma|}.
\]
The zeros of $A_n^{V_{\rm Ising}^{(1)}}(\lambda)$ can be interpreted as \emph{Fisher zeros} \cite{fisher1965lectures} of the Ising model on a \emph{random} $k$-regular graph. Here, by random we mean a graph drawn according to a probability distribution proportional to $|\Aut(G)|^{-1}$. For $k=4$, this case has served as the running example in this paper.
Another natural way to specialise $W$ is to set $\J=1$ and to consider zeros in the fugacity parameter $\lambda = \h$, so we set 
\[
  V_{\rm Ising}^{(2)}(x_1,x_2,\lambda) = \restr{W(x_1,x_2,\J,\h)}{\J = 1,\h=\lambda}.
\]
This brings us to the set-up of the classical Lee--Yang theorem, except that we are now dealing with the Ising model partition function of a random regular graph in the sense above. Indeed, among the critical points of $V_{\rm Ising}^{(2)}$ on $\Sc$ there are, for any $k\geq 3$, the two points $\sigma_1 = (\sqrt{2}/2, \sqrt{2}/2),~ \sigma_2 = (\sqrt{2}/2, -\sqrt{2}/2)$. After a straightforward computation, the corresponding anti-Stokes curve is seen to be the imaginary axis,
\[
  \left\{ \lambda\in\CC \,:\, |V_{\rm Ising}^{(2)}(\sigma_1,\lambda)| = |V_{\rm Ising}^{(2)}(\sigma_2,\lambda)| \right\} = \left\{ \lambda\in\CC \,:\, \Re \lambda = 0 \right\},
\]
in accordance with the classical Lee--Yang theorem.

\printbibliography

\noindent \textsc{Maximilian Wiesmann}\\
\textsc{Center for Systems Biology Dresden}\\
\textsc{Max Planck Institute of Molecular Cell Biology and Genetics}\\
\textsc{Max Planck Institute for the Physics of Complex Systems}\\
\url{wiesmann@pks.mpg.de}

\end{document}